\documentclass[oneside,english]{amsart}
\usepackage[T1]{fontenc}
\usepackage[latin9]{inputenc}
\usepackage{mathrsfs}
\usepackage{amsthm}
\usepackage{amsbsy}
\usepackage{amssymb}
\usepackage{graphicx}

\makeatletter
\numberwithin{equation}{section}
\numberwithin{figure}{section}
  \theoremstyle{definition}
  \newtheorem{defn}{\protect\definitionname}
  \theoremstyle{remark}
  \newtheorem*{rem*}{\protect\remarkname}
  \theoremstyle{plain}
  \newtheorem{prop}{\protect\propositionname}
  \theoremstyle{definition}
  \newtheorem*{example*}{\protect\examplename}
\theoremstyle{plain}
\newtheorem{thm}{\protect\theoremname}
  \theoremstyle{plain}
  \newtheorem{cor}{\protect\corollaryname}
  \theoremstyle{plain}
  \newtheorem{lem}{\protect\lemmaname}
 \theoremstyle{definition}
  \newtheorem{example}{\protect\examplename}

\usepackage{upgreek}

\usepackage{babel}
\providecommand{\definitionname}{Definition}
  \providecommand{\examplename}{Example}
  \providecommand{\lemmaname}{Lemma}
  \providecommand{\propositionname}{Proposition}
  \providecommand{\remarkname}{Remark}
\providecommand{\corollaryname}{Corollary}
\providecommand{\theoremname}{Theorem}

\makeatother

\usepackage{babel}
  \providecommand{\definitionname}{Definition}
  \providecommand{\examplename}{Example}
  \providecommand{\lemmaname}{Lemma}
  \providecommand{\propositionname}{Proposition}
  \providecommand{\remarkname}{Remark}
\providecommand{\corollaryname}{Corollary}
\providecommand{\theoremname}{Theorem}

\begin{document}

\title{Quantities, Dimensions and Dimensional Analysis}

\author{Dan Jonsson}

\address{Department of Sociology and Work Science, University of Gothenburg,
Box 720, SE 405 30 Göteborg, Sweden.}
\begin{abstract}
Formal definitions of quantities, quantity spaces, dimensions and
dimension groups are introduced. Based on these concepts, a theoretical
framework and a practical algorithm for dimensional analysis are developed,
and examples of dimensional analysis are given. 
\end{abstract}
\maketitle

\section{Introduction}

In a formula such as $E=mc^{2}$ or $\frac{\partial^{2}T}{\partial x^{2}}+\frac{\partial^{2}T}{\partial y^{2}}+\frac{\partial^{2}T}{\partial z^{2}}=0$,
used to express a physical law, the variables may be interpreted as
numerical measures of physical quantities. Such measures are somewhat
arbitrary. For example, the number $c$ representing the speed of
light in the numerical equation $E=mc^{2}$ depends not only on the
actual speed of light but also on the units of measurement used. This
suggests that it would be useful to represent physical quantities
(and other measurables) in a more definite and direct manner, as in
the modern approach to vector space theory, where vectors are seen
not as arrays of numbers but as abstract mathematical objects, which
can be referenced directly in a 'coordinate-free' manner by simple
symbols such as $u$, $v$ or $v'$. On the other hand, physical quantities
can be \emph{represented} by scalars through a process of 'coordinatisation'
\cite{key-8}: 
\[
physical\: quantity\rightarrow scalar\:(measure).
\]

It should be noted, though, that in a careful coordinatisation of
a physical or geometrical vector (characterized by its magnitude and
direction), numbers are seen as coordinates characterizing physical
phenomena \emph{indirectly} via mathematical objects, also called
vectors: 
\[
geometrical\: vector\rightarrow mathematical\: vector\rightarrow scalars\:(coordinates).
\]
 This suggests that one should introduce a more elaborate coordinatisation
of physical quantities: 
\[
physical\: quantity\rightarrow mathematical\: quantity\rightarrow scalar\:(measure),
\]
 where measures are similarly seen as characterizing physical quantities
indirectly.

As a point of departure for the analysis of quantities as abstract
mathematical objects, let us consider Maxwell's famous characterization
of quantities in his \emph{Treatise on Electricity and Magnetism}
\cite{key-6}, p. 41. 
\begin{quotation}
Every expression of a quantity consists of two factors or components.
One of these is the name of a certain known quantity of the same kind
as the quantity to be expressed, which is taken as a standard of reference.
The other component is the number of times the standard is to be taken
in order to make up the required quantity. The standard quantity is
technically called the Unit and the number is called the Numerical
value of the quantity 
\end{quotation}
Maxwell is essentially saying here that a quantity $q$ is something
which can be represented by a combination of (i) a number, the measure
of the quantity, and (ii) a unit of measurement, which is itself a
quantity of the same kind as $q$. Respect for Maxwell's genius should
not, however, make us blind to the fact that as a definition this
statement suffers from circularity. The notion of quantity presupposes
the notion of 'unit', but a 'unit' is itself a quantity ``of the
same kind'', so the notion of quantity is required in the first place.

One fairly obvious way of dealing with this circularity is to replace
the unit component of a quantity with a mathematical object of another
kind, related to but conceptually independent of a quantity: a \emph{sort},
such as 'meter', 'gallon' or 'hour'. A quantity is then seen as something
which can be represented by a combination of a scalar and a sort,
rather than a scalar and a unit-quantity; we can write scalar-sort
pairs as $\left(2,\mathsf{m}\right)$, $\left(3,\mathsf{kg}\right)$,
$\left(200,\mathsf{cm}\right)$ etc. It should be noted, though, that,
for example, $\left(2,\mathsf{m}\right)$ and $\left(200,\mathsf{cm}\right)$
are intended to represent the same physical quantity. Therefore, mathematical
quantities may instead be regarded as equivalence classes of scalar-sort
pairs characterizing the same physical quantities.

Another way of understanding and using Maxwell's observation is to
see it as a postulate about scalar multiplication. One of the axioms
defining vector spaces is that any vector $v$ can be multiplied by
a scalar $\lambda$, producing a vector $\lambda v$. Analogously,
one could, following Maxwell, take as axioms for an algebraic system
of quantities -- a ``quantity space'' -- the assumptions that (i)
the product $\lambda q$ of a scalar $\lambda$ and a quantity $q$
is a quantity $\lambda q$, and (ii) there is a set $U$ of quantities,
called units of measurement, such that for every quantity $q$ there
is some scalar $\lambda$ such that $q=\lambda u$ for some $u\in U$.

In this approach, circularity is not a concern, but (i) and (ii) are
not sufficient to set quantity spaces apart from related algebraic
structures such as vector spaces. It turns out to be possible to create
a suitable definition of quantity spaces by keeping scalar multiplication
(i) and adding other assumptions, however.

Specifically, quantities, as mathematical objects, can be defined
as concrete value-sort pairs (or equivalence classes of such pairs)
or as abstract mathematical objects. In the first case, operations
on quantities are defined in terms of operations on numbers and sorts;
in the second case, properties of operations on quantities are described
by means of abstract axioms. Such an abstract definition of quantity
spaces and quantities will be given below, noting that suitably defined
concrete mathematical quantities satisfy the given axioms for abstract
quantities.

The formal definition of mathematical quantities as elements of quantity
spaces given here makes it possible to formalize Maxwell's intuitive
notion of quantities ``of the same kind''. This is tantamount to
giving a formal definition of the notion of a \emph{dimension} such
as length or time. Dimensions will be defined as equivalence classes
of quantities; the equivalence classes of quantities in a particular
quantity space form a free abelian group, a so-called \emph{dimension
group}.

Quantities and quantity spaces are defined and discussed in Section
2, dimensions and dimension groups in Section 3. The remainder of
the article deals mainly with an important application of the mathematical
framework developed in these two sections, namely dimensional analysis.
Quantity functions and the representation of quantity functions by
means of scalar functions are treated in Section 4, principles and
methods of dimensional analysis are presented in Sections 5 and 6,
and examples of dimensional analysis are given in Section 7.

\section{On quantities}

\subsection{Abstract quantities}
\begin{defn}
\label{d2.1}Let $\mathbb{K}$ be a field. A \emph{scalar system}%
\footnote{The term \emph{semifield} is used in a similar sense.%
} $\mathcal{K}$ based on $\mathbb{K}$ is a subset of $\mathbb{K}$
such that $\mathcal{K}$ is closed under addition and the non-zero
elements of $\mathcal{K}$ form a group under multiplication. \emph{Scalars}
are elements of scalar systems. The unit element in $\mathcal{K}$
is denoted $1\!_{\mathcal{K}}$ or $1$. $\hphantom{\square}\square$ 
\end{defn}
In applications, $\mathbb{K}$ is usually the set of real numbers
$\mathbb{R}$, and the corresponding scalar system $\mathcal{R}$
is either $\mathbb{R}$ itself, the set $\mathbb{R}_{\geq0}$ of non-negative
numbers in $\mathbb{R}$, or the set $\mathbb{R}_{>0}$ of positive
numbers in $\mathbb{R}$. (The set $\mathbb{R}_{\neq0}$ of non-zero
real numbers is not closed under addition, although it is a group
under multiplication.)

Recall that a \emph{monoid} is a non-empty set together with an associative
binary operation and an identity element. We use multiplicative notation
for the binary operation, and the unit element of the monoid $Q$
will be denoted $1\!_{Q}$.
\begin{defn}
\label{d2.2}A\emph{ scalable (commutative) monoid} over a scalar
system $\mathcal{K}$ is a commutative monoid $Q$ such that there
is a function 
\[
\sigma:\mathcal{K}\times Q\rightarrow Q,\qquad\left(\alpha,q\right)\mapsto\sigma\left(\alpha,q\right)=\alpha q,
\]
called \emph{scalar multiplication}, such that for any $\alpha,\beta\in\mathcal{K}$
and any $q,q'\in Q$ we have 
\begin{enumerate}
\item $1\!_{\mathcal{K}}\, q=q$, 
\item $\alpha\left(\beta q\right)=\left(\alpha\beta\right)q$, 
\item $\alpha\left(qq'\right)=\left(\alpha q\right)q'.$ 
\end{enumerate}
An \emph{invertible} element $q\in Q$ is an element which has an
\emph{inverse} $q^{-1}\in Q$ such that $qq^{-1}=q^{-1}q=1\!_{Q}$.
We can define positive powers of $ $$q$, denoted $q^{c}$, in the
usual way; if $q$ is invertible, negative powers of $q$ can be defined
by setting $q{}^{c}=\left(q^{-1}\right)^{-c}$. By convention, $q^{0}=1\!_{Q}$.

A subset of a scalable monoid which is closed under monoid multiplication
and scalar multiplication is obviously also a scalable monoid, specifically
a \emph{scalable submonoid}. $\hphantom{\square}\square$
\end{defn}
The facts that $Q$ is commutative and associative and that scalar
multiplication is associative have some immediate consequences. For
example, $a(qq')=a(q'q)=(aq')q=q(aq')$, and $\left(\alpha q\right)\left(\alpha'q'\right)=\alpha\left(q\left(\alpha'q'\right)\right)=\alpha\left(\left(\alpha'q'\right)q\right)=\alpha\left(\alpha'\left(q'q\right)\right)=\alpha\left(\alpha'\left(qq'\right)\right)=\left(\alpha\alpha'\right)\left(qq'\right)$.
It is also clear that $q^{c}q^{d}=q^{\left(c+d\right)}$.
\begin{defn}
\label{d2.3}Let $Q$ be a scalable monoid over $\mathcal{K}$. A
(finite) \emph{basis} for $Q$ is a set $B=\left\{ b_{1},\ldots,b_{n}\right\} $
of invertible elements of $Q$ such that every $q\in Q$ has a unique
(up to order of factors) expansion 
\[
q=\mu\prod_{i=1}^{n}b_{i}^{_{_{k_{i}}}},
\]
where $\mu\in\mathcal{K}$ and $k_{1},\dots,k_{n}$ are integers.
$\hphantom{\square}\square$
\end{defn}
Any product of invertible quantities is invertible, so any product
of basis elements is invertible. If $\left\{ b_{1},\ldots,b_{n}\right\} $
is a basis for $Q$ then $\left\{ \lambda b_{1},\ldots,b_{n}\right\} $
is clearly a basis for $Q$ for any $\lambda\neq0$, since $\mu\prod_{i=1}^{n}b_{i}^{_{_{k_{i}}}}=\left(\mu/\lambda^{k_{1}}\right)\left(\lambda b_{1}\right)^{k_{1}}\prod_{i=2}^{n}b_{i}^{_{_{k_{i}}}}$.
\begin{defn}
\label{d2.4}A (finite-dimensional) \emph{free scalable monoid}, or
\emph{quantity space,} over $\mathcal{K}$ is a scalable monoid over
$\mathcal{K}$ which has a (finite) basis. The elements of a quantity
space are called \emph{quantities}. $\hphantom{\square}\square$
\end{defn}
It is clear that 
\[
\left(\mu\prod_{i=1}^{n}b_{i}^{_{_{k_{i}}}}\right)\left(\mu'\prod_{i=1}^{n}b_{i}^{_{k_{i}'}}\right)=\left(\mu\mu'\right)\prod_{i=1}^{n}b_{i}^{_{\left(k_{i}+k_{i}'\right)}}.
\]

\begin{rem*}
For a quantity space over $\mathcal{R}$, we can identify the scalar
product $\alpha\!\left(1_{Q}\right)$ with the real number $\alpha$,
so we could also define a scalable monoid over $\mathcal{R}$ as a
commutative monoid $Q$ such that $\mathcal{R}\subset Q$, $1_{Q}=1$,
and multiplication in $Q$ is consistent with the usual multiplication
in $\mathbb{\mathcal{R}}$, so that $1q=q$, since $1\alpha=\alpha,$
and $\alpha\left(\beta q\right)=\left(\alpha\beta\right)q,\,\alpha\left(qq'\right)=\left(\alpha q\right)q'$,
since $\alpha\left(\beta\gamma\right)=\left(\alpha\beta\right)\gamma$.
This approach has been used by, for example, Whitney \cite{key-7}
and Drobot \cite{key-4,key-5}. Partly in order to clarify the present
conceptualization, I will sketch a construction similar to the one
proposed by Drobot.

Recall that $\mathbb{R}$ can be regarded as a vector space over itself.
Similarly, $\mathbb{R}_{>0}$ is a vector space over $\mathbb{R}$
with \emph{vector addition} defined by $\alpha+\beta=\alpha\beta$
and \emph{scalar multiplication of a vector} defined by $\lambda\alpha=\alpha^{\lambda}$.
The bijection $\alpha\mapsto e^{\alpha}$ is an isomorphism between
$\mathbb{R}$ and $\mathbb{R}_{>0}$, and $\left\{ e\right\} $ is
a basis for $\mathbb{R}_{>0}$.

Let $V$ be an $n$-dimensional vector space over $\mathbb{R}$ and
consider the external direct sum of vector spaces $\Sigma=\mathbb{R}\oplus V$.
$\Sigma$ is a vector space over $\mathbb{R}$. There is a corresponding
vector space $\Pi=\mathbb{R}_{>0}\oplus V$ over $\mathbb{R}$. Using
additive notation, addition and scalar multiplication of vectors in
$\Pi$ are defined by 
\[
\left(\alpha,\mathbf{u}\right)+\left(\beta,\mathbf{v}\right)=\left(\alpha\beta,\mathbf{u}+\mathbf{v}\right),\qquad\lambda\left(\alpha,\mathbf{u}\right)=\left(\alpha^{\lambda},\lambda\mathbf{u}\right);
\]
using multiplicative notation, we can write 
\[
\left(\alpha,\mathbf{u}\right)\cdot\left(\beta,\mathbf{v}\right)=\left(\alpha\beta,\mathbf{u}\cdot\mathbf{v}\right),\qquad\left(\alpha,\mathbf{u}\right)^{\lambda}=\left(\alpha^{\lambda},\mathbf{u}^{\lambda}\right).
\]
(While $\mathbf{u}+\mathbf{v\mid\mathbf{u}\cdot\mathbf{v}}$ and $\lambda\mathbf{u}\mid\mathbf{u}^{\lambda}$
are alternative \emph{notations,} $\alpha+\beta\mid\alpha\beta$ and
$\lambda\alpha\mid\alpha^{\lambda}$ are different scalars.) The bijection
$\left(\alpha,\mathbf{u}\right)\mapsto\left(e^{\alpha},\mathbf{u}\right)$
is an isomorphism between $\Sigma$ and $\Pi$. As $V$ has a basis
$B$, every $\mathbf{v}\in\Pi$ has a unique expansion of the form
\[
\mathbf{v}=\lambda_{0}\left(e,\boldsymbol{0}\right)+\lambda_{1}\left(1,\mathbf{b}_{1}\right)+\ldots+\lambda_{n}\left(1,\mathbf{b}_{n}\right)\qquad\qquad\left(\lambda_{0},\ldots,\lambda_{n}\in\mathbb{R},\mathbf{b}_{i}\in B\right)
\]
in additive notation, and 
\[
\mathbf{v}=\left(e,\boldsymbol{1}\right)^{\lambda_{0}}\cdot\left(1,\mathbf{b}_{1}\right)^{\lambda_{1}}\cdot\ldots\cdot\left(1,\mathbf{b}_{n}\right)^{\lambda_{n}}\qquad\qquad\left(\lambda_{0},\ldots,\lambda_{n}\in\mathbb{R},\mathbf{b}_{i}\in B\right)
\]
in multiplicative notation, where the vector $\mathbf{0}$ is written
as $\mathbf{1}$. $\mathbf{v}\in\Sigma$ has a similar expansion,
and $\Sigma$ and $\Pi$ are $n+1$-dimensional vector spaces over
$\mathbb{R}$.

Drobot \cite{key-4} identifies $\left(\alpha,\mathbf{0}\right)$
or $\left(\alpha,\mathbf{1}\right)$ with the scalar $\alpha$, obtaining
analogues $\mathit{\Sigma}$ and $\mathit{\Pi}$ of $\Sigma$ and
$\Pi$, respectively, and then proves a variant of the so-called $\Pi$
theorem in dimensional analysis from facts about $\mathit{\Sigma}$
via corresponding facts about $\mathit{\Pi}$.

Instead of $\Pi$, let us consider a similar construction. Let $M$
be a module over $\mathbb{Z}$ and let $Q=\mathbb{R}_{>0}\oplus M$
be the external direct sum of the (abelian) multiplicative group of
$\mathbb{R}_{>0}$ and $M$ as an abelian group, so that again we
have 
\[
\left(\alpha,\mathbf{u}\right)+\left(\beta,\mathbf{v}\right)=\left(\alpha\beta,\mathbf{u}+\mathbf{v}\right)\quad\mathrm{or}\quad\left(\alpha,\mathbf{u}\right)\cdot\left(\beta,\mathbf{v}\right)=\left(\alpha\beta,\mathbf{u}\cdot\mathbf{v}\right).
\]

$Q$ is not a vector space over $\mathbb{R}$, but it is a commutative
monoid with $\left(1,\mathbf{1}\right)$ as unit element in multiplicative
notation, and we can define scalar multiplication on $Q$ as the function
\[
\sigma:\mathbb{R}_{>0}\times Q\rightarrow Q,\qquad\left(\lambda,\left(\alpha,\mathbf{u}\right)\right)\mapsto\lambda\left(\alpha,\mathbf{u}\right)=\left(\lambda\alpha,\mathbf{u}\right).
\]
It is easy to verify that then the conditions in Definition \ref{d2.2}
are satisfied, so $Q$ is a scalable monoid over $\mathbb{R}_{>0}$.

Furthermore, if $M$ has a finite basis $B=\left\{ \mathbf{b}_{1},\ldots,\mathbf{b}_{n}\right\} $
then every $q=\left(\alpha,\mathbf{u}\right)\in Q$ has a unique expansion
of the form 
\[
q=\left(\alpha,\mathbf{1}\right)\cdot\left(1,\mathbf{b}_{1}^{k_{1}}\cdot\ldots\cdot\mathbf{b}_{n}^{k_{n}}\right)\qquad\qquad\left(\alpha\in\mathbb{R}_{>0},k_{i}\in\mathbb{Z},\mathbf{b}_{i}\in B\right)
\]
in multiplicative notation, since $\left(\alpha,\mathbf{u}\right)$
has a unique decomposition of the form $\left(\alpha,\mathbf{u}\right)=\left(\alpha,\mathbf{1}\right)\cdot\left(1,\mathbf{u}\right)$.
The unique expansion of $q$ can be written as 
\[
q=\alpha\,\prod_{i=1}^{n}\left(1,\mathbf{b}_{i}\right)^{k_{i}}\qquad\qquad\left(\alpha\in\mathbb{R}_{>0},k_{i}\in\mathbb{Z},\mathbf{b}_{i}\in B\right),
\]
so $Q$ is a finite-dimensional quantity space over $\mathbb{R}_{>0}$
in the sense of Definition \ref{d2.4}, with $\left\{ \left(1,\mathbf{b}_{1}\right),\ldots,\left(1,\mathbf{b}_{n}\right)\right\} $
as a basis.

Note that $\mathbb{R}\oplus M$ can be defined as a quantity space
over $\mathbb{R}$ in the same way.
\end{rem*}

\subsection{Measures of quantities; invertible quantities}
\begin{defn}
\label{d2.5}Let $Q$ be a quantity space and let $B=\left\{ b_{1},\ldots,b_{n}\right\} $
be a basis for $Q$. The (uniquely determined) scalar $\mu$ in the
expansion 
\[
q=\mu\prod_{i=1}^{n}b_{i}^{k_{i}}
\]
is called the \emph{measure} of $q$ relative to \textbf{$B$} and
will be denoted by $\mu_{B}\!\left(q\right)$. If $\mu_{B}\!\left(q\right)$
does not depend on $B$, we may write $\mu_{B}\!\left(q\right)$ as
$\mu\!\left(q\right).$ $\hphantom{\square}\square$
\end{defn}
For example, we have $\mu_{B}\!\left(1_{Q}\right)=1$ for any $B$,
or $\mu\!\left(1_{Q}\right)=1$, because $1_{Q}=1\prod_{i=1}^{n}b_{i}^{0}$
for any $B$.
\begin{prop}
\label{s2.1}Let $B=\left\{ b_{1},\ldots,b_{n}\right\} $ be a basis
for a quantity space $Q$. 
\begin{enumerate}
\item For any $q\in Q$, $\mu_{B}\!\left(\lambda q\right)=\lambda\mu_{B}\!\left(q\right)$. 
\item For any $q,q'\in Q$, $\mu_{B}\!\left(qq'\right)=\mu_{B}\!\left(q\right)\mu_{B}\!\left(q'\right)$. 
\item A quantity $q\in Q$ is invertible if and only if $\mu_{B}\!\left(q\right)\neq0$,
and $\mu_{B}\!\left(q^{-1}\right)=\mu_{B}\!\left(q\right)^{-1}=1/\mu_{B}\!\left(q\right)$.
\end{enumerate}
\end{prop}
\begin{proof}
Set $q=\mu p$, where $p=\prod_{i=1}^{n}b{}_{i}^{k_{i}}$, and $q'=\mu'p'$,
where $p'=\prod_{i=1}^{n}b_{i}^{k_{i}'}$.\\
(1). $\mu_{B}\!\left(\lambda q\right)=\mu_{B}\!\left(\lambda\left(\mu p\right)\right)=\mu_{B}\!\left(\left(\lambda\mu\right)p\right)=\lambda\mu=\lambda\mu_{B}\!\left(q\right)$.
\\
(2). $qq'=\left(\mu p\right)\left(\mu'p'\right)=\left(\mu\mu'\right)\left(pp'\right)$,
so $\mu_{B}\!\left(qq'\right)=\mu\mu'=\mu_{B}\!\left(q\right)\mu_{B}\!\left(q'\right)$.
\\
(3). If $\mu_{B}\!\left(q\right)\neq0$ we have 
\[
\frac{\prod_{i=1}^{n}b_{i}^{-k_{i}}}{\mu_{B}\!\left(q\right)}q=q\frac{\prod_{i=1}^{n}b_{i}^{-k_{i}}}{\mu_{B}\!\left(q\right)}=\left(\mu_{B}\!\left(q\right)\prod_{i=1}^{n}b_{i}^{k_{i}}\right)\frac{\prod_{i=1}^{n}b_{i}^{-k_{i}}}{\mu_{B}\!\left(q\right)}=1_{Q},
\]
so $q$ is invertible. If, conversely, $q$ has an inverse $q^{-1}$
then $\mu_{B}\!\left(q\right)\mu_{B}\!\left(q^{-1}\right)=\mu_{B}\!\left(1_{Q}\right)=1$,
so $\mu_{B}\!\left(q\right)\neq0$ and $\mu_{B}\!\left(q^{-1}\right)=1/\mu_{B}\!\left(q\right)$.$\qedhere$ \end{proof}
\begin{prop}
\label{s2.2}If $q\in Q$ is invertible and $\lambda q=\lambda'q$
then $\lambda=\lambda'$.\end{prop}
\begin{proof}
Set $q=\mu\prod_{i=1}^{n}b{}_{i}^{k_{i}}$. If $\lambda q=\lambda'q$
then $\lambda\mu\prod_{i=1}^{n}b_{i}^{k_{i}}=x=\lambda'\mu\prod_{i=1}^{n}b_{i}^{k_{i}}$,
so $\lambda\mu=\lambda'\mu$, since the representation of $x$ is
unique, and $\mu\neq0$ since $q$ is invertible, so $\lambda=\lambda'$.
$\qedhere$
\end{proof}

\subsection{Modeling and representation of quantities by sort-assigned scalars}

As briefly mentioned in the Introduction, scalar-sort pairs, or equivalence
classes of such pairs, may be used to model physical quantities. On
the other hand, such constructs can also be used to represent abstract
mathematical quantities in much the same way as tuples of scalars
represent abstract vectors. It is worthwhile to elaborate somewhat
on these observations.

\emph{(a)}. Let $S^{\star}$ be the free abelian group on $S=\left\{ s_{1},\ldots,s_{n}\right\} $,
and let $\mathcal{K}$ be a scalar system. We define multiplication
on $P=\mathcal{K}\times S^{\star}$ by $\left(k,s\right)\left(k',s'\right)=\left(kk',ss'\right)$,
while scalar multiplication of an element $\left(k,s\right)$ of $P$
by an element $\lambda$ of $\mathcal{K}$ is defined by $\lambda\left(k,s\right)=\left(\lambda k,s\right)$.

It is easy to verify that with these operations $P$ is a scalable
monoid over $\mathcal{K}$. Furthermore, as $S$ is a basis for $S^{\star}$,
$\left\{ \left(1,s\right)\mid s\in S\right\} $ is a basis for $P$,
so $P$ is a quantity space over $\mathcal{K}$. In particular, we
can let $\mathcal{K}$ be a scalar system of real numbers $\mathcal{R}$,
and interpret the elements of $S$ as \emph{sorts} such as 'kilogram',
'centimeter', 'second', 'meter per second' etc. With this interpretation
in mind, we may call $\left(k,s\right)\in P$ a \emph{sort-assigned
scalar} and $P$ a \emph{space of sort-assigned scalars}.

\emph{(b)}. Let $P$ be a space of sort-assigned scalars over $\mathcal{K}$,
and let $\sim$ be a congruence relation on $P$. Then we can define
the product of equivalence classes in $P/\!\sim$ without ambiguity
by setting $\left[p\right]\left[q\right]=\left[pq\right]$ for any
$p,q\in P$. We can also set $\lambda\left[q\right]=\left[\lambda q\right]$
for any $\lambda\in\mathcal{K},q\in P$. It can be shown that $P/\!\sim$
with these operations is also a quantity space over $\mathcal{K}$.

\emph{(c)}. Suppose that there is some \emph{universe of measurables},
in geometry, physics and engineering called \emph{physical quantities}.
A particular sort-assigned scalar may \emph{characterize} more than
one measurable. For example, $\left(10,\mathsf{cm}\right)$ may characterize
both the base and the height of a triangle.

On the other hand, different sort-assigned scalars may characterize
the same physical quantities. For example, with the usual interpretation
$\left(100,\mathsf{cm}\right)$ and $\left(1,\mathsf{m}\right)$ characterize
the same measurables. Let $\sim$ be a relation on $P$ such that
$p\sim q$ if and only if $p$ and $q$ characterize the same measurables.
This is obviously an equivalence relation on $P$, so there is a corresponding
set $P/\!\sim$ of equivalence classes. Let us assume that if $p$
and $p'$ characterize the same measurables and $q$ and $q'$ characterize
the same measurables then (i) $\lambda p$ and $\lambda p'$ characterize
the same measurables for every $\lambda\in\mathcal{K}$, and (ii)
$pq$ and $p'q'$ characterize the same measurables. Then $\sim$
is a congruence relation, so $P/\!\sim$ is a quantity space.

\emph{(d)}. We may denote the equivalence class of sort-assigned scalars
which contains $\left(c,s\right)$ by $\left[c,s\right]$. Alternatively,
we may write quantities such as $\left[1,\mathsf{kg}\right]$, $\left[2,\mathsf{m}^{2}\right]$,
$\left[3,\mathsf{ms^{-1}}\right]$ in the more familiar forms $1\,\mathsf{kg}$,
$2\,\mathsf{m^{2}},$ $3\,\mathsf{m/s}$ etc. Thus, we interpret an
equation such as $100\,\mathsf{cm}=1\,\mathsf{m}$ as $\left[100,\mathsf{cm}\right]=\left[1,\mathsf{m}\right]$
rather than $\mbox{\ensuremath{\left(100,\mathsf{cm}\right)}=\ensuremath{\left(1,\mathsf{m}\right)}}$\linebreak{}
 -- the latter equality is false, since $100\neq1$ and $\mathsf{cm}\neq\mathsf{m}$.

\subsection{Mathematical quantities as values of physical quantities}

The principal \emph{raison d'être} of systems of (abstract or concrete)
mathematical quantities is that they can be used to model systems
of measurables, in particular so-called physical quantities, which
are roughly speaking measurable properties of physical objects and
systems. We establish connections between mathematical and physical
quantities by means of names such as $radius$ and $circum\! f\! erence$
(of a circle), $length$ and $width$ (of a rectangle) etc. For example,
in a modeling context $radius$ refers to a particular physical quantity
as well as a particular mathematical quantity; a mathematical quantity
is thus linked to a physical quantity if and only if they have the
same name. Note that while distinct names always refer to distinct
physical quantities, distinct names can refer to the same mathematical
quantity. For example, the length and the width of a rectangle are
distinct physical quantities, but if the rectangle is a square then
the length and the width of the rectangle is the same mathematical
quantity. The mathematical quantity corresponding to a physical quantity
can be thought of as the \emph{value} of that physical quantity.

We can use common names of physical and mathematical quantities to
define new physical and mathematical quantities. For example, given
the physical and mathematical quantities $arc\; length$ and $radius$
we can define a physical and mathematical quantity $radian$ such
that the mathematical quantity $radian$ is connected to the mathematical
quantities $arc\: length$ and $radius$ by the relation 
\[
radian=arc\; length\cdot radius^{-1}.
\]
One should not think of such formulas as directly involving physical
quantities, however; it is not possible to multiply or divide physical
quantities as such.

\subsection{Systems of units of measurement}

Units of measurement and systems of such units are notions which are
linked to notions of physical and mathematical quantities. While a
full theory of units of measurement requires consideration of both
physical and mathematical quantities, some basic notions can be described
in terms of mathematical quantities, specifically abstract quantity
spaces.

A \emph{system of direct units of measurements} for a subset $S$
of a quantity space $Q$ is a set $U$ of quantities in $S$ such
that for every $q\in S$ there is some $u\in U$ and some $\mu\in\mathcal{K}$,
uniquely determined by $q$ and $u$, such that 
\[
q=\mu u.
\]
A \emph{system of fundamental units of measurement} for $Q$ is a
basis for $Q$; that is, a set $B=\left\{ b_{1},\ldots,b_{n}\right\} $
such that every $q\in Q$ has a unique expansion 
\[
q=\mu\prod_{i=1}^{n}b_{i}^{k_{i}}.
\]
In particular, every unit of measurement $u$ in any system of direct
units of measurement $U$ has a unique expansion of this kind. A \emph{coherent
system of units of measurement} for $S$ is a system of direct units
of measurement $U$ such that there is a system of fundamental units
of measurement $B$ for $Q$ such that every $u\in U$ has a unique
expansion of the form 
\[
u=1\prod_{i=1}^{n}b_{i}^{k_{i}}.
\]
The set of all $u\in Q$ of this form is obviously a coherent system
of units of measurement for all of $Q$. Relative to a coherent system
of units of measurement $U$, given by a basis $B$, every $q\in S$
has a representation of the form $q=\mu u$, where $\mu\in\mathcal{K}$
and $u\in U$ are uniquely determined by $q,$ relative to $B$.
\begin{rem*}
It is not always emphasized that all quantities in a quantity space
are defined in terms of fundamental units in a unique way, meaning
that the fundamental units do not only generate the quantity space
but also form a basis. Nevertheless, the fundamental units in systems
of measurement such as the CGS system, the MKS system or the SI system
invariably form a basis.
\end{rem*}

\section{On dimensions}

\subsection{Equidimensional quantities and dimensions}
\begin{defn}
\label{d3.1}Let $Q$ be a quantity space over $\mathcal{R}$, and
let $\sim$ be a relation on $Q$ such that $q\sim q'$ if and only
if $\alpha q=\beta q'$ for some $\alpha,\beta\in\mathcal{R}.$ Then
$q$ and $q'$ are said to be \emph{equidimensional} quantities, and
the relation $\sim$ is accordingly said to be an \emph{equidimensionality
relation}. $\hphantom{\square}\square$ 
\end{defn}
As a trivial consequence of this definition, $q\sim\lambda q$ for
any $q\in Q$ and $\lambda\in\mathcal{R}$, since $\lambda q=1\left(\lambda q\right)$.
\begin{prop}
\label{s3.1}An equidimensionality relation on a quantity space is
an equivalence relation.\end{prop}
\begin{proof}
The equidimensionality relation $\sim$ is reflexive because $1q=1q$,
it is symmetric because if $\alpha q=\beta q'$ then $\beta q'=\alpha q$,
and it is transitive because if $\alpha q=\beta q'$ and $\gamma q'=\delta q''$
then $\alpha\gamma q=\beta\gamma q'$ and $\beta\gamma q'=\beta\delta q''$,
so $\alpha\gamma q=\beta\delta q''$.$\qedhere$ \end{proof}
\begin{defn}
\label{d3.2}A \emph{dimension} is an equivalence class of quantities
with respect to equidimensionality. We denote the dimension containing
the quantity $q$ by $\left[q\right]$. $\hphantom{\square}\square$\end{defn}
\begin{prop}
\label{s3.2}Let $Q$ be a quantity space. The equidimensionality
relation $\sim$ is a congruence relation on $Q$, and $Q/{\sim}$
is a commutative monoid, with the product in $Q/{\sim}$ defined by
\[
[q][q']=[qq']
\]
for any $q,q'\in Q$.\end{prop}
\begin{proof}
If $\alpha q=\beta x$ and $\alpha'q'=\beta'x'$ then $(\alpha q)(\alpha'q')=(\beta x)(\beta'x')$,
so $(\alpha\alpha')(qq')=(\beta\beta')(xx')$. Thus, if $q\sim x$
and $q'\sim x'$ then $qq'\sim xx'$, $ $so the equivalence relation
$\sim$ is a congruence on $Q$, and the product in $Q/{\sim}$ defined
by $[q][q']=[qq']$ does not depend on the choice of representative
for $\left[q\right]$ and $\left[q'\right]$.

Straight-forward calculations show that $Q/{\sim}$ is associative
and commutative and that $\left[\mathbf{1}\right]q=q\left[\mathbf{1}\right]=q$.
$\qedhere$\end{proof}
\begin{example*}
In a quantity space over $\mathcal{R}$ whose elements are equivalence
classes of sort-assigned scalars, we have $5\left(2\,\mathsf{m}\right)=10\,\mathsf{m}=1\left(10\,\mathsf{m}\right)$,
so $2\,\mathsf{m}\sim10\,\mathsf{m}$, and $0\left(1\,\mathsf{s}\right)=0\,\mathsf{s}=1\left(0\,\mathsf{s}\right)$,
so $1\,\mathsf{s}\sim0\,\mathsf{s}$. On the other hand, there are
no $\alpha,\beta\in\mathcal{R}$ such that $\alpha\left(2\,\mathsf{m}\right)=\beta\left(2\,\mathsf{g}\right)$,
$\alpha\left(2\,\mathsf{m}\right)=\beta\left(0\,\mathsf{g}\right)$
or $\alpha\left(0\,\mathsf{m}\right)=\beta\left(0\,\mathsf{g}\right)$,
meaning that $2\,\mathsf{m}\nsim2\,\mathsf{g}$, $2\,\mathsf{m}\nsim0\,\mathsf{g}$
and $0\,\mathsf{m}\nsim0\,\mathsf{g}$. 
\end{example*}
Let $q$ be a quantity in $Q$ and $\mathfrak{d}$ a dimension in
$Q/{\sim}$. If $\left[q\right]=\mathfrak{d}$, or equivalently $q\in\mathfrak{d}$,
we say that $q$\emph{ has dimension }$\mathfrak{d}$, or that the
dimension of $q$ is $\mathfrak{d}$. The way the concept of 'dimension'
is defined here is generally consistent with previous informal and
formal uses of this term. For example, consider the principle of ``dimensional
homogeneity'' frequently invoked when dealing with physical quantities
(or other measurables), namely that quantities cannot be equal if
they do not have the same dimension. In terms of the present understanding
of dimensions, this principle is just the following simple fact about
quantities and dimensions: 
\[
\mathrm{if}\; q=q'\;\mathrm{then}\;\left[q\right]=\left[q'\right],\quad\mathrm{or}\;\mathrm{equivalently,}\quad\mathrm{if}\;\left[q\right]\neq\left[q'\right]\;\mathrm{then}\; q\neq q'.
\]
\pagebreak{}
\begin{prop}
\label{s3.3a}Let $q,q'\in Q$ have the expansions $\mu_{B\!}\left(q\right)\!\prod_{i=1}^{n}\! b_{i}^{k_{i}}$,
\textup{$\mu_{B\!}\left(q'\right)\!\prod_{i=1}^{n}\! b_{i}^{k_{i}'}$}
relative to a basis $B=\left\{ b_{1},\ldots,b_{n}\right\} $ for $Q$.
If $q\sim q'$ then $k_{i}=k_{i}'$ for $i=1,\ldots,n$.\end{prop}
\begin{proof}
By assumption, there are scalars $\alpha,\beta$ such that 
\[
\left(\alpha\,\mu_{B\!}\left(q\right)\right)\prod_{i=1}^{n}b_{i}^{k_{i}}=x=\left(\beta\,\mu_{B\!}\left(q'\right)\right)\prod_{i=1}^{n}b_{i}^{k_{i}'},
\]
and since the expansion of $x$ relative to $B$ is unique, $k_{i}=k_{i}'$
for $i=1,\ldots,n$.$\qedhere$ \end{proof}
\begin{prop}
\label{s3.3}If $q,q'\in Q$ and $q\sim q'$ then $\mu_{B}\!\left(q'\right)q=\mu_{B}\!\left(q\right)q'$.\end{prop}
\begin{proof}
Use Proposition \ref{s3.3a} and note that

\[
\mu_{B}\!\left(q'\right)\left(\mu_{B}\!\left(q\right)\prod_{i=1}^{n}b_{i}^{k_{i}}\right)=\mu_{B}\!\left(q\right)\left(\mu_{B}\!\left(q'\right)\prod_{i=1}^{n}b_{i}^{k_{i}}\right).\qedhere
\]

\end{proof}
Conversely, if $\mu_{B}\!\left(q'\right)q=\mu_{B}\!\left(q\right)q'$
then $q\sim q'$ by definition, so Proposition \ref{s3.3} gives an
alternative definition of equidimensionality. Note that if $\mu_{B}\!\left(q\right)\neq0$
then $\mu_{B}\!\left(q'\right)q=\mu_{B}\!\left(q\right)q'$ implies
\[
q'=\frac{\mu_{B}\!\left(q'\right)}{\mu_{B}\!\left(q\right)}q,
\]
so Proposition \ref{s3.3} shows that if $q\sim q'$ and $q$ is invertible
then there is some $\lambda$ such that $q'=\lambda q$; if $q'$
is invertible as well then $\mu_{B}\!\left(q'\right)\neq0$, so $\lambda\neq0$.
\begin{thm}
\label{s3.4}Let $Q$ be a quantity space over $\mathcal{R}.$ For
every $q\in\left[1_{Q}\right]$, $\mu_{B}\!\left(q\right)$ does not
depend on $B$.\end{thm}
\begin{proof}
As $1_{Q}\sim q$, $\mu_{B}\!\left(q\right)1_{Q}=\mu_{B}\!\left(1_{Q}\right)q$
by Proposition \ref{s3.3}, so if $\mu_{B}\!\left(q\right)\neq\mu_{B'}\!\left(q\right)$,
so that $\mu_{B}\!\left(q\right)1_{Q}\neq\mu_{B'}\!\left(q\right)1_{Q}$
by Proposition \ref{s2.2}, then $\mu_{B}\!\left(1_{Q}\right)\neq\mu_{B'}\!\left(1_{Q}\right)$,
but this contradicts the fact that $\mu_{B}\!\left(1_{Q}\right)=1$
for any $B$.$\qedhere$ 
\end{proof}
So-called \emph{dimensionless} quantities are often \emph{defined,}
in line with Theorem \ref{s3.4}, as those whose measures do not depend
on the basis (system of fundamental units of measurement) chosen;
\emph{dimensional} or \emph{dimensionful} quantities are then defined
as non-dimensionless quantities. This terminology is not consistent
with how quantities and dimensions are conceptualized here, however,
since a ``dimensionless'' quantity does have a dimension, namely
$\left[1_{Q}\right]$. I prefer the terms \emph{quasiscalar quantity}
and \emph{proper quantity}.

Note that despite Theorem \ref{s3.4}, a ``dimensionless'' quantity
can have different measures relative to different ``dimensionless''
\emph{direct} units of measurement. For example, angles can be measured
in both radians and degrees, although angles are ``dimensionless'',
but there is only one possible unit of measurement of an angle in
a \emph{coherent} system of (direct) units of measurement defined
in terms of a system of fundamental units (that is, a basis for $Q$),
namely the radian.

\subsection{Sums of quantities}

Let $p\in Q$ be an invertible quantity. As $\mathcal{R}$ is closed
under addition, we can define the sum $q+_{\! p}q'$ of $q=\lambda p$
and $q'=\lambda'p$ relative to $p$ by setting 
\[
q+_{\! p}q'=\left(\lambda+\lambda'\right)p.
\]
If $p'$ is another invertible quantity such that $p'\sim p$, Proposition
\ref{s3.3} implies that $p'=\kappa p$ for some non-zero $\kappa\in\mathcal{R}$.
Hence, 
\[
q+_{\! p'}q'=\left(\lambda/\kappa\right)\kappa p+_{\!\kappa p}\left(\lambda'/\kappa\right)\kappa p=\left(\lambda/\kappa+\lambda'/\kappa\right)\kappa p=\left(\lambda+\lambda'\right)p=q+_{\! p}q',
\]
so $q+_{\! p}q'$ does not depend on $p$. Note that if $q=\lambda p$
and $q'=\lambda'p$ then $\lambda'q=\lambda q'$, so $q\sim q'$;
conversely, if $q\sim q'$ then there is an invertible quantity $p$
such that $q=\lambda p$ and $q'=\lambda'p$ by Proposition \ref{s3.3a}.
Thus, $q+q'$ can be defined in the following way if and only if $q\sim q'$:
\begin{defn}
\label{d3.3}Let $Q$ be a quantity space over $\mathcal{R}$. We
set 
\[
\lambda p+\lambda'p=\left(\lambda+\lambda'\right)p
\]
for any invertible $p\in Q$ and $\lambda,\lambda'\in\mathcal{R}$.
$\hphantom{\square}\square$
\end{defn}
As a trivial consequence of this definition, addition of quantities
is commutative because addition of scalars is commutative. Straight-forward
calculation also shows that addition of quantities is associative
because addition of scalars is associative.
\begin{prop}
\label{s3.5}If $q,q'\in Q$, $q\sim q'$ and B is a basis for $Q$
then $\mu_{B}\!\left(q\right)+\mu_{B}\!\left(q'\right)=\mu_{B}\!\left(q+q'\right)$.\end{prop}
\begin{proof}
Let $q=\mu_{B}\!\left(q\right)\prod_{i=1}^{n}b_{i}^{k_{i}}$ and $q'=\mu_{B\!}\left(q'\right)\prod_{i=1}^{n}b_{i}^{k_{i}}$
be the expansions of $q$ and $q'$ relative to $B=\left\{ b_{1},\ldots.b_{n}\right\} $.
$\prod_{i=1}^{n}b_{i}^{k_{i}}$ is invertible, so 
\[
q+q'=\mu_{B}\!\left(q\right)\prod_{i=1}^{n}b_{i}^{k_{i}}+\mu_{B\!}\left(q'\right)\prod_{i=1}^{n}b_{i}^{k_{i}}=\left(\mu_{B}\!\left(q\right)+\mu_{B\!}\left(q'\right)\right)\prod_{i=1}^{n}b_{i}^{k_{i}},
\]
so we have obtained the unique expansion of $q+q'$ relative to $B$,
and this expansion shows that $\mu_{B}\!\left(q+q'\right)=\mu_{B}\!\left(q\right)+\mu_{B}\!\left(q'\right)$.$\qedhere$ 
\end{proof}
If $0\in\mathcal{R}$ then there is for each $q\in Q$ a quantity
$0q\in Q$ such that 
\[
q+0q=0q+q=q,
\]
because if $q$ has the expansion $q=\mu\prod_{i=1}^{n}b_{i}^{k_{i}}$
relative to $B$ then $0q=0\prod_{i=1}^{n}b_{i}^{k_{i}}$. In fact,
$0p=0\prod_{i=1}^{n}b_{i}^{k_{i}}=0q$ for any $p\sim q$, so $0q$
is the unique zero quantity in $\left[q\right]$.

If $\mathcal{R}=\mathbb{R}$, we set 
\[
-q=\left(-1\right)q\qquad\mathrm{and}\qquad q'-q=q'+\left(-q\right).
\]
for any $q\in Q$. Then we have 
\[
q-q=-q+q=0q
\]
for any $q\in Q$.

\subsection{Dimensions as vector spaces}
\begin{prop}
\label{p8}Let $Q$ be a quantity space over $\mathcal{R}$. For any
$\alpha,\alpha'\in\mathcal{R}$ and $q,q'\in Q$ we have
\begin{enumerate}
\item $\alpha\left(q+q'\right)=\alpha q+\alpha q'$;
\item $\left(\alpha+\alpha'\right)q=\alpha q+\alpha'q$.
\end{enumerate}
\end{prop}
\begin{proof}
Set $q=\lambda p$ and $q'=\lambda'p$, where $p\in Q$ is invertible.
\\
(1). $\alpha\left(q+q'\right)=\alpha\left(\lambda p+\lambda'p\right)=\alpha\left(\lambda+\lambda'\right)p=\left(\alpha\lambda+\alpha\lambda'\right)p=\alpha\lambda p+\alpha\lambda'p$\\
$\phantom{(1).}=\alpha q+\alpha q'$;\\
(2). $\left(\alpha+\alpha'\right)q=\left(\alpha+\alpha'\right)\lambda p=\left(\alpha\lambda+\alpha'\lambda\right)p=\alpha\lambda p+\alpha'\lambda p=\alpha q+\alpha'q$.
$\qedhere$
\end{proof}
We conclude that a dimension $\mathfrak{d}$ in a quantity space over
$\mathbb{R}$, or indeed any field, can be regarded as a one-dimensional
vector space. As we have seen, an additive group structure, corresponding
to the additive group of the scalar field, can be defined on $\mathfrak{d}$.
Furthermore, in view of (1) and (2) in Definition \ref{d2.2} and
Proposition \ref{p8}, vector space scalar multiplication is defined
on $\mathfrak{d}$. Finally, the fact that there is some $p\in\mathfrak{d}$
such that for every $q\in\mathfrak{d}$ there is some scalar $\lambda$
such that $q=\lambda p$ means that $\mathfrak{d}$ is a one-dimensional
vector space.

\subsection{Dimension groups}

Since $Q$ is commutative, $Q/{\sim}$ is also commutative, and in
this section we prove that $Q/{\sim}$ is not only a commutative monoid
but in fact a free abelian group.
\begin{prop}
\label{s3.6}If $Q$ is a quantity space then $Q/{\sim}$ is an abelian
group.\end{prop}
\begin{proof}
For every $\mathfrak{d}\in Q/{\sim}$ there is some $q\in Q$ such
that $\mathfrak{d}=\left[q\right]$. In view of Proposition \ref{s3.2},
suffices to show that $\mathfrak{d}$ has an inverse. Let $B=\left\{ b_{1},\ldots,b_{n}\right\} $
be a basis for $Q$, so that $q=\mu\prod_{i=1}^{n}b_{i}^{k_{i}}$,
and set $x=\prod_{i=1}^{n}b_{i}^{k_{i}}.$ Then $x$ is invertible,
and $\left[q\right]=\left[x\right]$, since $q=\mu x$. We have $\left[x^{-1}\right]\mathfrak{d}=\mathfrak{d}\left[x^{-1}\right]=\left[q\right]\left[x^{-1}\right]=\left[x\right]\left[x^{-1}\right]=\left[xx^{-1}\right]=\left[1_{Q}\right]$,
so $\left[x^{-1}\right]$ is the inverse of $\mathfrak{d}$.$\qedhere$\end{proof}
\begin{thm}
\label{s3.7}Let $Q$ be a quantity space. 
\begin{enumerate}
\item If $B=\left\{ b_{1},\ldots,b_{n}\right\} $ is a basis for $Q$, then
$B^{*}=\left\{ \left[b_{1}\right],\ldots,\left[b_{n}\right]\right\} $
is a basis for $Q/{\sim}$ with the same number of elements. 
\item If $B^{*}=\left\{ \left[b_{1}\right],\ldots,\left[b_{n}\right]\right\} $,
where each $b_{i}$ is invertible, is a basis for $Q/{\sim}$, then
$B=\left\{ b_{1},\ldots,b_{n}\right\} $ is a basis for $Q$ with
the same number of elements. 
\end{enumerate}
\end{thm}
\begin{proof}
(1). The unique expansions of $b_{i},b_{i'}\in B$ relative to $B$
are $b_{i}=1b_{i}$ and $b_{i'}=1b_{i'}$, so $\mu_{B}\left(b_{i}\right)=\mu_{B}\left(b_{i'}\right)=1$.
Hence, $\left[b_{i}\right]=\left[b_{i'}\right]$ implies $b_{i}=b_{i'}$
according to Proposition \ref{s3.3}, so the mapping $b_{i}\mapsto\left[b_{i}\right]$
is one-to-one.

Let $\mathfrak{d}=\left[q\right]$ be an arbitrary dimension in $Q/{\sim}$.
As $B$ generates $Q$, $q=\mu\prod_{i=1}^{n}b_{i}^{k_{i}}$ for some
integers $k_{1},\ldots,k_{n}$, so $\mathfrak{d}=\left[\mu\prod_{i=1}^{n}b_{i}^{k_{i}}\right]=\left[\prod_{i=1}^{n}b_{i}^{k_{i}}\right]=\prod_{i=1}^{n}\left[b_{i}\right]^{k_{i}}$,
so $B^{*}$ generates $Q/{\sim}$.

Also, if $\mathfrak{d}=\prod_{i=1}^{n}\left[b_{i}\right]^{k_{i}}=\prod_{i=1}^{n}\left[b_{i}\right]^{k_{i}'}$,
then $\left[1\prod_{i=1}^{n}b_{i}^{k_{i}}\right]=\left[1\prod_{i=1}^{n}b_{i}^{k_{i}'}\right]$,
so\linebreak{}
 $k_{i}=k_{i}'$ for $i=1,\ldots,n$ by Proposition \ref{s3.3a},
since $B$ is a basis for $Q$.\medskip{}

(2). If $b_{i}=b_{i'}$ then $\left[b_{i}\right]=\left[b_{i'}\right]$,
so the mapping $\left[b_{i}\right]\mapsto b_{i}$ is one-to-one.

Consider an arbitrary $q\in Q$. We have $\left[q\right]=\prod_{i=1}^{n}\left[b_{i}\right]^{k_{i}}=\left[\prod_{i=1}^{n}b_{i}^{k_{i}}\right]$,
so \linebreak{}
 $q\sim\prod_{i=1}^{n}b_{i}^{k_{i}}$. $\prod_{i=1}^{n}b_{i}^{k_{i}}$
is invertible, so Proposition \ref{s3.3} implies that there exists
some $\lambda$ such that $q=\lambda\prod_{i=1}^{n}b_{i}^{k_{i}}$.

Finally, if $q=\mu\prod_{i=1}^{n}b_{i}^{k_{i}}=\mu'\prod_{i=1}^{n}b_{i}^{k_{i}'}$
then $\left[\mu\prod_{i=1}^{n}b_{i}^{k_{i}}\right]=\left[\mu'\prod_{i=1}^{n}b_{i}^{k_{i}'}\right]$,
so $\left[\prod_{i=1}^{n}b_{i}^{k_{i}}\right]=\left[\prod_{i=1}^{n}b_{i}^{k_{i}'}\right]$,
so $\prod_{i=1}^{n}\left[b_{i}\right]^{k_{i}}=\prod_{i=1}^{n}\left[b_{i}\right]^{k_{i}'}$,
so $k_{i}=k_{i}'$ for $i=1,\ldots,n$, since $B^{*}$ is a basis
for $Q/\sim$. Also, $q=\mu x=\mu'x$, where $x$ is invertible, so
$\mu=\mu'$ by Proposition \ref{s2.2}.$\qedhere$\end{proof}
\begin{cor}
\label{s3.8}For any (finite-dimensional) quantity space $Q$, the
quotient monoid $Q/{\sim}$ is a free abelian group (of finite rank).
\pagebreak{}
\end{cor}
\begin{cor}
\label{s3.10}Any two finite bases for $Q/{\sim}$ have the same number
of elements, and any two finite bases for $Q$ have the same number
of elements.\end{cor}
\begin{proof}
Any two finite bases for a free abelian group have the same cardinality.$\qedhere$\end{proof}
\begin{defn}
\label{d3.4}A \emph{dimension group} is a subgroup of a free abelian
group $Q/{\sim}$. $\hphantom{\square}\square$\end{defn}
\begin{cor}
\label{s3.9}Any dimension group is free abelian.\end{cor}
\begin{proof}
Any subgroup of a free abelian group is free abelian.$\qedhere$
\end{proof}

\subsection{Dependent and independent dimensions and quantities}

Since dimension groups are free abelian groups, many notions from
group theory are applicable, and I shall introduce some concepts that
will be used later.
\begin{defn}
\label{d3.5}Let $G$ be a dimension group, and let $\mathfrak{S}=\left\{ \mathfrak{d}_{1},\ldots,\mathfrak{d}_{n}\right\} $
be a set of dimensions in $G$. $\mathfrak{d}\in G$ is said to be
\emph{dependent} on $\mathfrak{S}$ (or on $\mathfrak{d}_{1},\ldots,\mathfrak{d}_{n}$)
if and only if there are integers $k\neq0,k_{1},\dots,k_{n}$ such
that 
\[
\mathfrak{d}^{k}=\prod_{i=1}^{n}\mathfrak{d}_{i}^{k_{i}},
\]
where by convention $\prod_{i=1}^{0}\mathfrak{d}_{i}^{k_{i}}=\left[1_{Q}\right]$,
and $\left[1_{Q}\right]$ is dependent on the empty set of dimensions.
We may assume that $k>0$ without loss of generality.

Also, $\mathfrak{S}$ is said to be a set of \emph{independent} dimensions
if and only if no $\mathfrak{\mathfrak{d}}_{i}\in\mathfrak{S}$ is
dependent on $\mathfrak{S}-\left\{ \mathfrak{d}_{i}\right\} $, or
equivalently if and only if 
\[
\prod_{i=1}^{n}\mathfrak{d}_{i}^{k_{i}}=\left[1_{Q}\right]
\]
implies that $k_{i}=0$ for $i=1,\ldots,n$. Finally, a set of dimensions
$\mathfrak{S}\subset\mathfrak{T}$ is said to be a \emph{maximal}
set of independent dimensions in $\mathfrak{T}$ if and only if $\mathfrak{S}$
is a set of independent dimensions and any $\mathfrak{d}_{i}\in\mathfrak{T}-\mathfrak{S}$
is dependent on $\mathfrak{S}$. $\hphantom{\square}\square$ 
\end{defn}
Note that $\mathfrak{S}$ is a maximal set of independent dimensions
in $\mathfrak{T}$ if and only if $\mathfrak{S}$ is a maximal set
of independent dimensions in the dimension group $G$ generated by
$\mathfrak{T}$, so that every $\mathfrak{d}\in G$ is dependent on
$\mathfrak{S}$.

Recall that any maximal set of independent elements in a free abelian
group $G$, in particular a dimension group, has the same number $r$
of elements, called the \emph{rank} of $G$. Every basis for a free
abelian group is obviously a maximal set of independent elements in
$G$, but a maximal set $S$ of independent elements in $G$ does
not necessarily generate $G$, so $S$ is a basis if and only if $S$
generates $G$.

The notions of dependence and (maximal) sets of independent elements
introduced above can be defined for quantity spaces as well. Consider
a quantity space $Q$ over $\mathcal{R}$; a quantity $q\in Q$ is
said to be \emph{(dimensionally) dependent} on a set of invertible
quantities $\left\{ q_{1},\ldots,q_{n}\right\} $ in $Q$ if and only
if there are integers $k>0,k_{1},\dots,k_{n}$ and some $\lambda\in\mathcal{R}$
such that 
\[
q^{k}=\lambda\prod_{i=1}^{n}q_{i}^{k_{i}},
\]
where by convention $\prod_{i=1}^{0}q_{i}^{k_{i}}=1_{Q}$, and $1_{Q}$
is dependent on the empty set of quantities.

Using this notion of a dependent quantity, we can define a set of
independent invertible quantities and a maximal set of independent
invertible quantities in the same way as the corresponding concepts
for dimensions. Equivalently, $\left\{ q_{1},\ldots,q_{n}\right\} $
is a set of independent invertible quantities in $Q$ if and only
if 
\[
\lambda\prod_{i=1}^{n}q{}_{i}^{k_{i}}=1_{Q},
\]
implies that $k_{i}=0$ for $i=1,\ldots,n$ (and $\lambda=1$).

When proving Theorem \ref{s3.7}, we also proved that if $q_{1},\ldots,q_{n}$
are invertible then $q_{1},\ldots,q_{n}$ are independent if and only
if $\left[q_{1}\right],\ldots,\left[q_{n}\right]$ are independent.
The following broader statement can be proved similarly. 
\begin{thm}
\label{s3.11}Let $Q$ be a quantity space, and let $q_{1},\ldots,q_{n}$
be invertible quantities in $Q$. Then $q\in Q$ is dependent on $q_{1},\ldots,q_{n}$
if and only if $\left[q\right]$ is dependent on $\left[q_{1}\right],\ldots,\left[q_{n}\right]$,
and $\left\{ q_{1},\ldots,q_{n}\right\} $ is a (maximal) set of independent
quantities in $Q$ if and only if $\left\{ \left[q_{1}\right],\ldots,\left[q_{n}\right]\right\} $
is a (maximal) set of independent dimensions in $Q/\!\sim$. 
\end{thm}

\section{Quantity functions and scalar representations}

The so-called laws of nature and other mathematically described empirical
regularities typically involve relations between quantities -- relations
which can be expressed by means of quantity functions. As noted in
the Introduction, the well-known equation $E=mc^{2}$ may be seen
as expressing a relation between three real numbers which are the
measures of three quantities relative to some system of units of measurement,
but the symbols $E$, $m$ and $c$ can also be seen as directly representing
these three quantities, so that $E=mc^{2}$ is interpreted as a quantity
relation rather than a relation between scalars. This quantity relation
has the form 
\[
q=\Phi\!\left(q_{1},\ldots,q_{n}\right),
\]
where $q,q_{1},\dots,q_{n}$ belong to some quantity space $Q$. These
variables do not range over the entire quantity space, though; each
variable takes values only within a subset of $Q$, namely a dimension
in $Q/{\sim}$. In $E=mc^{2}$, for example, $E$ has dimension 'energy',
$m$ has dimension 'mass', and $c$ has dimension 'velocity'.

\subsection{Basic notions}
\begin{defn}
\label{d4.1}Let $Q$ be a quantity space over $\mathcal{R}$. A function
\[
\Phi:\mathfrak{D}_{1}\times\ldots\times\mathfrak{D}_{n}\rightarrow\mathfrak{D},\quad\left(q_{1},\ldots,q_{n}\right)\mapsto q,
\]
where $\mathfrak{D},\mathfrak{D}_{1},\ldots,\mathfrak{D}_{n}\in Q/{\sim}$,
is called a \emph{(dimensional) quantity function} on $Q$.

A \emph{trivial} quantity function is a quantity function of the form
\[
\left(q_{1},\ldots,q_{n}\right)\mapsto0q.
\]

A \emph{quasiscalar} quantity function is a quantity function of the
form 
\[
\left[1_{Q}\right]\times\ldots\times\left[1_{Q}\right]\rightarrow\left[1_{Q}\right].
\]

A \emph{monomial} quantity function is a quantity function of the
form 
\[
\left(q_{1},\ldots,q_{n}\right)\mapsto\lambda q_{1}^{c_{1}}\cdots q_{n}^{c_{n}},
\]
where $\lambda\in\mathcal{K}$ and $c_{1},\dots.c_{n}$ are integers.$\hphantom{\square}\square$ 
\end{defn}
It is important to be clear about the difference between quantity
functions and scalar functions used to represent quantity functions. 
\begin{defn}
\label{d4.2}Let $Q$ be a quantity space over $\mathcal{R}$, and
let $\Phi:\mathfrak{D}_{1}\times\ldots\times\mathfrak{D}_{n}\rightarrow\mathfrak{D}$
be a quantity function on $Q$. A \emph{scalar representation} of
$\Phi$ relative to a basis $B$ for $Q$ is a function 
\[
\phi_{B}:\mathcal{R}^{n}\rightarrow\mathcal{R},\quad\left(s_{1},\dots,s_{n}\right)\mapsto s
\]
such that 
\[
\phi_{B}\!\left(\mu_{B}\!\left(q_{1}\right),\dots,\mu_{B}\!\left(q_{n}\right)\right)=\mu_{B}\!\left(\Phi\!\left(q_{1},\dots,q_{n}\right)\right)
\]
for any $q_{1},\dots,q_{n}$. (Note that for any dimension $\mathfrak{d}\in Q$
and any basis $\left\{ b_{1},\ldots,b_{n}\right\} $ for $Q$ there
are integers $k_{1},\ldots,k_{n}$ such that $\mu\mapsto\mu\prod_{i=1}^{n}q_{i}^{k_{i}}$
is a bijection $\mathcal{R}\rightarrow\mathfrak{d}$.) 

If, in particular, the scalar representation of $\Phi$ relative to
a basis $B$ is the same for any $B$, we write $\phi_{B}$ as $\phi$,
and we have 
\[
\phi\!\left(\mu_{B}\!\left(q_{1}\right),\dots,\mu_{B}\!\left(q_{n}\right)\right)=\mu_{B}\!\left(\Phi\!\left(q_{1},\dots,q_{n}\right)\right)
\]
for any $q_{1},\dots,q_{n}$ and any $B$. In this case, $\phi$ is
said to be a \emph{covariant} scalar representation of $\Phi$, while
$\Phi$ is said to be \emph{covariantly representable}.$\hphantom{\square}\square$\end{defn}
\begin{prop}
\label{s4.1}Let $Q$ be a quantity space over $\mathcal{R},$ let
$\Phi$ be a quantity function $\mathfrak{D}\rightarrow\mathfrak{D}$
on $Q$ defined by $\Phi\!\left(q\right)=q$ for every $q\in\mathfrak{D}$,
and let $\phi$ be a scalar function $\mathcal{R}\rightarrow\mathcal{R}$
defined by $\phi\!\left(s\right)=s$ for every $s\in\mathcal{R}$.
Then $\phi$ is the covariant scalar representation of $\Phi$.\end{prop}
\begin{proof}
By definition, $\phi\!\left(\mu_{B}\!\left(q\right)\right)=\mu_{B}\!\left(q\right)=\mu_{B}\!\left(\Phi\!\left(q\right)\right)$
for any $B$.$\qedhere$\end{proof}
\begin{prop}
\label{s4.2}Any quasiscalar function has a covariant scalar representation.\end{prop}
\begin{proof}
If $q,q_{i}\in\left[1_{Q}\right]$ then $\mu_{B}\!\left(q\right),\mu_{B}\!\left(q_{i}\right)$
do not depend on $B$.$\qedhere$\end{proof}
\begin{cor}
\label{s4.3}Let $\psi:\mathcal{R\rightarrow R}$ be a scalar function
and let $\Phi_{\psi}:\left[1_{Q}\right]\rightarrow\left[1_{Q}\right]$
be the quantity function defined by $\Phi_{\psi}\!\left(\lambda1_{Q}\right)=\psi\!\left(\lambda\right)1_{Q}$.
Then $\Phi_{\psi}$ has a covariant scalar representation, namely
$\psi$. 
\end{cor}
Thus we may identify any scalar function $\psi$ with a corresponding
quasiscalar function $\Phi_{\psi}$, and every quasiscalar function
$\Phi$ has the form $\lambda1_{Q}\mapsto\psi_{\Phi}\!\left(\lambda\right)\!1_{Q}$
for some scalar function $\psi_{\Phi}$.

\subsection{Scalar representations of composite quantity functions}

Let \linebreak{}
 $\Phi:\mathfrak{D}_{1}\times\ldots\times\mathfrak{D}_{n}\rightarrow\mathfrak{D}$
and $\Psi:\mathfrak{D}_{1}\times\ldots\times\mathfrak{D}_{n}\rightarrow\mathfrak{D}$
be quantity functions and define $\lambda\Phi$ by $\lambda\Phi\!\left(q_{1},.\dots,q_{n}\right)=\lambda\left(\Phi\!\left(q_{1},.\dots,q_{n}\right)\right)$,
and $\Phi+\Psi$ by $\left(\Phi+\Psi\right)\!\left(q_{1},.\dots,q_{n}\right)=\Phi\!\left(q_{1},.\dots,q_{n}\right)+\Psi\!\left(q_{1},.\dots,q_{n}\right)$.
Also, let scalar products $\lambda\phi$ and sums $\phi+\psi$ of
scalar functions be defined in the same way.
\begin{prop}
\label{s4.4}Let $\phi$ and $\psi$ be the covariant scalar representations
of $\Phi$ and $\Psi$, respectively. Then (1) $\lambda\phi$ is the
covariant scalar representation of $\lambda\Phi$, and (2) $\phi+\psi$
is the covariant scalar representation of $\Phi+\Psi$.\end{prop}
\begin{proof}
Set $\mu_{B}\!\left(q_{i}\right)=s_{i}$.\\
(1). $\lambda\phi\!\left(s_{1},.\dots,s_{n}\right)=\lambda\left(\phi\!\left(s_{1},\ldots,s_{n}\right)\right)=\lambda\left(\mu_{B}\!\left(\Phi\!\left(q_{1},.\dots,q_{n}\right)\right)\right)$\\
$\phantom{(1).}=\mu_{B}\!\left(\lambda\left(\Phi\!\left(q_{1},.\dots,q_{n}\right)\right)\right)=\mu_{B}\!\left(\lambda\Phi\!\left(q_{1},.\dots,q_{n}\right)\right)$.\\
(2). $\left(\phi+\psi\right)\!\left(s_{1},.\dots,s_{n}\right)=\phi\!\left(s_{1},.\dots,s_{n}\right)+\psi\!\left(s_{1},.\dots,s_{n}\right)$\\
$\phantom{(1).}=\mu_{B}\!\left(\Phi\!\left(q_{1},\ldots.q_{n}\right)\right)+\mu_{B}\!\left(\Psi\!\left(q_{1},\ldots.q_{n}\right)\right)=\mu_{B}\!\left(\Phi\!\left(q_{1},\ldots.q_{n}\right)+\Psi\!\left(q_{1},\ldots.q_{n}\right)\right)$\\
$\phantom{(1).}=\mu_{B}\!\left(\left(\Phi+\Psi\right)\!\left(q_{1},\ldots.q_{n}\right)\right)\qedhere$.
\end{proof}
Now let $\Phi:\mathfrak{D}_{1}\times\ldots\times\mathfrak{D}_{n}\rightarrow\mathfrak{D}$
and $\Psi:\mathfrak{D}_{1}'\times\ldots\times\mathfrak{D}_{m}'\rightarrow\mathfrak{D'}$
be quantity functions, let $\Phi\Psi$ be defined by $\Phi\Psi\!\left(p_{1},.\dots,p_{n},q_{1},\dots,q_{m}\right)=\Phi\!\left(p_{1},.\dots,p_{n}\right)\Psi\!\left(q_{1},\dots,q_{m}\right)$,
and let $\Phi^{-1}$ be defined by $\Phi^{-1}\!\left(p_{1},.\dots,p_{n}\right)=\Phi\!\left(p_{1},.\dots,p_{n}\right)^{-1}=1_{Q}/\Phi\!\left(p_{1},.\dots,p_{n}\right)$.
Also, let products and inverses of scalar functions be defined in
the same way.
\begin{prop}
\label{s4.5}If $\phi$ is the covariant scalar representation of
$\Phi$ and $\psi$ is the covariant representation of $\Psi$ then
(1) $\phi\psi$ is the covariant scalar representation of $\Phi\Psi$
and (2) $\phi^{-1}$ is the covariant scalar representation of $\Phi^{-1}$.\end{prop}
\begin{proof}
Set $\mu_{B}\!\left(p_{i}\right)=s_{i}$ and $\mu_{B}\!\left(q_{i}\right)=t_{i}$.\\
(1). $\phi\psi\!\left(s_{1},\dots,s_{n},t_{1},\dots,t_{m}\right)=\phi\!\left(s_{1},\dots,s_{n}\right)\psi\!\left(t_{1},\dots,t_{m}\right)$\\
$\phantom{(1).}=\mu_{B}\!\left(\Phi\!\left(p_{1},\dots,p_{n}\right)\right)\mu_{B}\!\left(\Psi\!\left(q_{1},\dots,q_{m}\right)\right)=\mu_{B}\!\left(\Phi\!\left(p_{1},\dots,p_{n}\right)\Psi\!\left(q_{1},\dots,q_{m}\right)\right)$\\
$\phantom{(1).}=\mu_{B}\!\left(\Phi\Psi\!\left(p_{1},\dots,p_{n},q_{1},\dots,q_{m}\right)\right)$.\\
(2). $\phi^{-1}\!\left(s_{1},\dots,s_{n}\right)=1/\phi\!\left(s_{1},\dots,s_{n}\right)=1/\mu_{B}\!\left(\Phi\!\left(q_{1},\dots,q_{n}\right)\right)$\\
$\phantom{(1).}=\mu_{B}\!\left(1_{Q}/\Phi\!\left(q_{1},\dots,q_{n}\right)\right)=\mu_{B}\!\left(\Phi^{-1}\!\left(q_{1},\dots,q_{n}\right)\right).\qedhere$
\end{proof}
If $n=m$ and $\mathfrak{D}_{i}=\mathfrak{D}_{i}'$ for $i=1,\ldots,n$
in the definitions of $\Phi$ and $\Psi$, we can set 
\[
\Phi\Psi\!\left(p_{1},.\dots,p_{n}\right)=\Phi\!\left(p_{1},.\dots,p_{n}\right)\Psi\!\left(p_{1},\dots,p_{n}\right)
\]
and define covariant scalar representations $\phi:\mathcal{R}^{n}\rightarrow R$
of $\Phi$ and $\psi:\mathcal{R}^{n}\rightarrow\mathcal{R}$ of $\Psi$
accordingly. $\phi\psi$ is again the covariant scalar representation
of $\Phi\Psi$.
\begin{prop}
\label{s4.8}A monomial function $\left(q_{1},\ldots,q_{n}\right)\mapsto\lambda q_{1}^{k_{1}}\cdot\cdots\cdot q_{n}^{k_{n}}$
has the covariant scalar representation $\left(s_{1},\ldots,s_{n}\right)\mapsto\lambda s_{1}^{k_{1}}\cdot\cdots\cdot s_{n}^{k_{n}}$. \end{prop}
\begin{proof}
Immediate from Propositions \ref{s4.1}, \ref{s4.4} and \ref{s4.5}.$\qedhere$\end{proof}
\begin{prop}
\label{s4.6}If $\phi$ and $\psi$ are the covariant scalar representations
of\linebreak{}
 $\Phi:\mathfrak{D}_{1}\rightarrow\mathfrak{D}_{0}$ and $\Psi:\mathfrak{D}_{2}\rightarrow\mathfrak{D}_{1}$,
respectively, then $\phi\circ\psi$ is the covariant scalar representation
of $\Phi\circ\Psi$.\end{prop}
\begin{proof}
By definition, $\phi\circ\psi\left(\mu_{B}\!\left(q\right)\right)=\phi\!\left(\psi\!\left(\mu_{B}\!\left(q\right)\right)\right)=\phi\!\left(\mu_{B}\!\left(\Psi\!\left(q\right)\right)\right)=\mu_{B}\!\left(\Phi\!\left(\Psi\!\left(q\right)\right)\right)$\\
 $=\mu_{B}\!\left(\Phi\circ\Psi\left(q\right)\right)$.$\qedhere$ 
\end{proof}
Consider quantity functions 
\[
\Phi:\mathfrak{D}_{1}\times\ldots\times\mathfrak{D}_{n}\rightarrow\mathfrak{D}_{0},\quad\left(q_{1},\ldots,q_{n}\right)\mapsto q
\]
and 
\begin{gather*}
\Psi_{1}:\mathfrak{D}_{11}\times\ldots\times\mathfrak{D}_{1m_{1}}\rightarrow\mathfrak{D}_{1},\quad\left(q_{11},\ldots,q_{1m_{1}}\right)\mapsto q_{1},\\
\vdots\\
\Psi_{n}:\mathfrak{D}_{n1}\times\ldots\times\mathfrak{D}_{nm_{n}}\rightarrow\mathfrak{D}_{n},\quad\left(q_{n1},\ldots,q_{nm_{n}}\right)\mapsto q_{n}.
\end{gather*}
Define the quantity function $\Phi\circ\left(\Psi_{1},\ldots,\Psi_{n}\right)$
by 
\[
\Phi\circ\left(\Psi_{1},\ldots,\Psi_{n}\right)\left(q_{11},\ldots,q_{nm_{n}}\right)=\Phi\!\left(\Psi_{1}\!\left(q_{11},\ldots,q_{1m_{1}}\right),\ldots,\Psi_{n}\!\left(q_{n1},\ldots,q_{nm_{n}}\right)\right).
\]
Also consider scalar functions $\phi:\mathcal{R}^{n}\rightarrow\mathcal{R}$,
$\psi_{1}:\mathcal{R}^{m_{1}}\rightarrow\mathcal{R},\ldots,\psi_{n}:\mathcal{R}^{m_{n}}\rightarrow\mathcal{R}$,
and define $\phi\circ\left(\psi_{1},\ldots,\psi_{n}\right)$ similarly.
The following more general result can be proved in the same way as
Proposition \ref{s4.6}.
\begin{prop}
\label{s4.7}If $\phi$ and $\psi_{1},\ldots,\psi_{n}$ are the covariant
scalar representations of $\Phi$ and $\Psi_{1},\ldots,\Psi_{n}$,
respectively, then $\phi\circ\left(\psi_{1},\ldots,\psi_{n}\right)$
is the covariant scalar representation of $\Phi\circ\left(\Psi_{1},\ldots,\Psi_{n}\right)$. 
\end{prop}

\section{The representation theorems for quantity functions}

\subsection{Statement and proof of the representation theorems}

In this section, two representation theorems for quantity functions,
one of them directly corresponding to the so-called $\Pi$ theorem
in dimensional analysis, will be stated and proved. Dimensional analysis
is based on the principle that 'laws of nature' can be numerically
represented without reference to arbitrarily chosen units of measurement
\cite{key-2,key-3}. In other words, quantity functions that represent
'laws of nature' have covariant scalar representations. Quantity functions
which are subject to this restriction have special properties, and
their scalar representations also have special properties, as expressed
by the representation theorems presented below. The following result,
which I call Barenblatt's lemma because it is based on ideas from
\cite{key-2}, is the crucial step in the derivation of the representation
theorems.
\begin{thm}
\label{s5.1}(Barenblatt's lemma). Let $Q$ be a quantity space over
$\mathcal{R}$, let \linebreak{}
$\mathfrak{D},\mathfrak{D}_{1},\ldots,\mathfrak{D}_{m},\mathfrak{B}_{1},\ldots,\mathfrak{B}_{r}$
be dimensions in $Q/{\sim}$ such that $\left\{ \mathfrak{B}_{1},\ldots,\mathfrak{B}_{r}\right\} $
is a basis for the dimension group generated by $\mathfrak{D}_{1},\ldots,\mathfrak{D}_{m},\mathfrak{B}_{1},\ldots,\mathfrak{B}_{r}$,
and consider the non-trivial quantity function 
\[
\Psi:\mathfrak{D}_{1}\times\cdots\times\mathfrak{D}_{m}\times\mathfrak{B}_{1}\times\cdots\times\mathfrak{B}_{r}\rightarrow\mathfrak{D},\qquad\left(p_{1},\dots,p_{m},q_{1},\dots,q_{r}\right)\mapsto p.
\]

\begin{enumerate}
\item If $\Psi$ has a covariant scalar representation then $\mathfrak{D}$
is dependent on $\mathfrak{B}_{1},\ldots,\mathfrak{B}_{r}$.
\item If in addition $\mathfrak{D},\mathfrak{D}_{1},\ldots,\mathfrak{D}_{m}=\left[1_{Q}\right]$
then there is a quasiscalar function 
\[
\Phi:\mathfrak{D}_{1}\times\cdots\times\mathfrak{D}_{m}\rightarrow\mathfrak{D},\qquad\left(p_{1},\dots,p_{m}\right)\mapsto p
\]
such that 
\[
p=\Psi\!\left(p_{1},\dots,p_{m},q_{1},\dots,q_{r}\right)=\Phi\!\left(p_{1},\dots,p_{m}\right)
\]
for any $p_{i}\in\mathfrak{D}_{i}$ and any invertible $q_{i}\in\mathfrak{B}_{i}$. 
\end{enumerate}
\end{thm}
\begin{proof}
(1). Consider a quantity relation 
\begin{equation}
p=\Psi\!\left(p_{1},\ldots,p_{m},q_{1},\dots,q_{r}\right),\label{eq:qrel-0}
\end{equation}
where $\mu_{\mathcal{B}}\!\left(p\right)\neq0$. Let $B=\left\{ b,b_{1},\ldots,b_{r}\right\} $,
where $b\in\mathfrak{D},b_{1}\in\mathfrak{B}_{1},\ldots,b_{r}\in\mathfrak{B}_{r}$,
be a set of invertible quantities in $Q$. If $\mathfrak{D},\mathfrak{B}_{1},\dots,\mathfrak{B}_{r}$
are independent dimensions in $Q/{\sim}$, then $b,b_{1},\ldots,b_{r}$
are independent quantities in $Q$. Let $Q'$ be the subspace of $Q$
generated by $B$ so that $B$ is a basis for $Q'$. Note that $B'=\left\{ \lambda b,b_{1},\ldots,b_{r}\right\} $,
where $\lambda>1$, is another basis for $Q'$.

We clearly have $\mu_{B}\!\left(q_{i}\right)=\mu_{B'}\!\left(q_{i}\right)$
for $i=1,\ldots,r$. Also, by assumption there are integers $c_{ij}$
such that $\mathfrak{D}_{i}=\prod_{j=1}^{m}\mathfrak{B}_{j}^{c_{ij}}$
for $i=1,\ldots,m$, so there are scalars $\mu_{i}\in\mathcal{R}$
such that $p_{i}=\mu_{i}\prod_{j=1}^{n}b_{j}^{c_{ij}}$ for $i=1,\ldots,m$,
so $\mu_{B}\!\left(p_{i}\right)=\mu_{i}\prod_{j=1}^{r}\mu_{B}\!\left(b_{j}\right)^{c_{ij}}=\mu_{i}\prod_{j=1}^{r}\mu_{B'}\!\left(b_{j}\right)^{c_{ij}}=\mu_{B'}\!\left(p_{i}\right)$
for $i=1,\ldots,m$. Hence, the scalar relations corresponding to
(\ref{eq:qrel-0}) relative to $B$ and $B'$, respectively, are 
\begin{gather*}
\mu_{B}\!\left(p\right)=\psi{}_{B}\!\left(s_{1},\ldots,s_{m},t_{1},\ldots,t_{r}\right),\\
\mu_{B'}\!\left(p\right)=\mu_{B}\!\left(p\right)/\lambda=\psi_{B'}\!\left(s_{1},\ldots,s_{m},t_{1},\ldots,t_{r}\right),
\end{gather*}
where $s_{i}=\mu_{B}\!\left(p_{i}\right)=\mu_{B'}\!\left(p_{i}\right)$
and $t_{i}=\mu_{B}\!\left(q_{i}\right)=\mu_{B'}\!\left(q_{i}\right)$.
As $\mu_{B}\!\left(p\right)\neq0$ and $\lambda>1$ so that $\mu_{B}\!\left(p\right)\neq\mu_{B}\!\left(p\right)/\lambda$,
we conclude that $\psi{}_{B}\neq\psi{}_{B'}$.

Conversely, if $\Psi$ has a covariant scalar representation then
$\mathfrak{D},\mathfrak{B}_{1},\dots,\mathfrak{B}_{r}$ are not independent,
and as $\mathbf{\mathfrak{B}_{1},\dots,\mathfrak{B}_{r}}$ are independent,
$\mathfrak{D}$ is dependent on $\mathbf{\mathfrak{B}_{1},\dots,\mathfrak{B}_{r}}$,
meaning that there are integers $k>0,k_{1},\ldots,k_{r}$ such that
$\mathfrak{D}^{k}=\prod_{i=1}^{r}\mathfrak{B}_{i}^{k_{i}}$.

(2). Consider two quantity relations 
\begin{equation}
p=\Psi\!\left(p_{1},\ldots,p_{m},q_{1},q_{2},\ldots,q_{r}\right),\label{eq:qrel-1}
\end{equation}
\begin{equation}
p'=\Psi\!\left(p_{1},\ldots,p_{m},q_{1}',q_{2},\ldots,q_{r}\right),\label{eq:qrel-2}
\end{equation}
where $q_{1}\neq q_{1}'$. Let $B=\left\{ b_{1},\ldots,b_{r}\right\} $,
where $b_{1}\in\mathfrak{B}_{1},\ldots,b_{r}\in\mathfrak{B}_{r}$,
be a set of invertible quantities in $Q$. Since $\mathfrak{B}_{1},\dots,\mathfrak{B}_{r}$
are independent dimensions in $Q/{\sim}$, $b_{1},\ldots,b_{r}$ are
independent quantities in $Q$. Let $Q'$ be the subspace of $Q$
generated by $B$, so that $B$ is a basis for $Q'$. Let $\lambda,\lambda'$
be defined by $q_{1}=\lambda b_{1},q_{1}'=\lambda'b_{1}$. As $q_{1}$
and $q_{1}'$ are invertible, $\lambda,\lambda'\neq0$. Thus, $B'=\left\{ \left(\lambda'/\lambda\right)b_{1},b_{2},\dots,b_{r}\right\} $
is another basis for $Q'$.

Clearly, $\mu_{B}\!\left(q_{1}\right)=\mu_{B'}\!\left(q_{1}'\right)=\lambda$.
It is also clear that $\mu_{B}\!\left(q_{i}\right)=\mu_{B'}\!\left(q_{i}\right)$
for $i=2,\ldots,r$, and $\mu_{B}\!\left(p_{i}\right)=\mu_{B'}\!\left(p_{i}\right)$
for $i=1,\ldots,m$ since $p_{i}\in\left[1_{Q}\right]$. Thus the
scalar representations of (\ref{eq:qrel-1}) relative to $B$ and
(\ref{eq:qrel-2}) relative to $B'$, respectively, are
\begin{gather*}
\mu_{B}\!\left(p\right)=\psi\left(s_{1},\ldots,s_{m}\lambda,t_{2},\ldots,t_{r}\right),\\
\mu_{B'}\!\left(p'\right)=\psi\left(s_{1},\ldots,s_{m},\lambda,t_{2},\ldots,t_{r}\right),
\end{gather*}
where $s_{i}=\mu_{B}\!\left(p_{i}\right)=\mu_{B'}\!\left(p_{i}\right)$
and $t_{i}=\mu_{B}\!\left(q_{i}\right)=\mu_{B'}\!\left(q_{i}\right)$,
and we have used the fact that $\Psi$ has a covariant scalar representation
$\psi$. Thus, $\mu_{B}\!\left(p\right)=\mu_{B'}\!\left(p'\right)$,
and by assumption $\left[p\right]=\left[p'\right]=\left[1_{Q}\right]$.
Hence, $\mu_{B'}\!\left(p'\right)=\mu_{B}\!\left(p'\right)$, so $\mu_{B}\!\left(p\right)=\mu_{B}\!\left(p'\right)$
and therefore $p=\mu_{B}\!\left(p\right)1_{Q}=\mu_{B}\!\left(p'\right)1_{Q}=p'$.

This means that $p$ does not depend upon $q_{1}$ with $q_{2},\ldots,q_{r}$
fixed, and it is shown similarly that $p$ does not depend on $q_{i}$
with $\ldots,q_{i-1},q_{i+1},\ldots$ fixed. Hence, 
\begin{gather*}
\Psi\!\left(p_{1},\ldots,p_{m},q_{1},\ldots,q_{n}\right)=\Psi\!\left(p_{1},\ldots,p_{m},1_{Q},q_{2},\ldots\right)=\Psi\!\left(p_{1},\ldots,p_{m},1_{Q},1_{Q},q_{3},\ldots\right)\\
=\cdots=\Psi\!\left(p_{1},\ldots,p_{m},1_{Q},\ldots,1_{Q}\right)=\Phi\!\left(p_{1},\ldots,p_{m}\right).\qedhere
\end{gather*}
\end{proof}
\begin{lem}
\label{s5.4}Let $Q$ be a quantity space over $\mathbb{R}_{>0}$
and $\mathfrak{D}$ a dimension in $Q$. If $\mathfrak{p}\in\mathfrak{D}^{k}$
for some integer $k>0$, then there is a unique $p\in\mathfrak{D}$
such that $\mathfrak{p}=p^{k}$. \end{lem}
\begin{proof}
Let $\left\{ b_{1},\ldots,b_{n}\right\} $ be a basis for $Q$. If
$\mathfrak{D}=\left[q\right]$ then $\mathfrak{D}^{k}=\left[q\right]^{k}=\left[q^{k}\right]$,
so if $q=\mu\prod_{i=1}^{n}b_{i}^{c_{i}}$ then $\mathfrak{D}^{k}=\left[q^{k}\right]=\left[\mu^{k}\prod_{i=1}^{n}b_{i}^{kc_{i}}\right]=\left[\prod_{i=1}^{n}b_{i}^{kc_{i}}\right]$,
so if $\mathfrak{p}\in\mathfrak{D}^{k}$ then $\mathfrak{p}=\nu\prod_{i=1}^{n}b_{i}^{kc_{i}}=\left(\sqrt[k]{\nu}\prod_{i=1}^{n}b_{i}^{c_{i}}\right)^{k}$,
where $\sqrt[k]{\nu}\prod_{i=1}^{n}b_{i}^{c_{i}}\in\left[q\right]=\mathfrak{D}$.$\qedhere$ 
\end{proof}
$p$ is said to be the $k$th root of $\mathfrak{p}$, denoted $\sqrt[k]{\mathfrak{p}}$.
Proposition \ref{s4.7} implies that if $p^{k}$ has a unique root
$\sqrt[k]{p^{k}}$ and the quantity function $\Psi:\left(q_{1},\ldots,q_{n}\right)\mapsto p^{k}$
has the covariant scalar representation $\psi$, then $\Phi$ defined
by $\Phi\left(q_{1},\ldots,q_{n}\right)=\sqrt[k]{\Psi\left(q_{1},\ldots,q_{n}\right)}$
has the covariant scalar representation $\phi$ given by $\phi\left(s_{1},\ldots,s_{n}\right)=\sqrt[k]{\psi\left(s_{1},\ldots,s_{n}\right)}$. 
\begin{thm}
\label{s5.3}(Reformulated $\Pi$ theorem.) Let $Q$ be a quantity
space over $\mathbb{R}_{>0}$, let $\mathfrak{D},\mathfrak{D}_{1},\ldots,\mathfrak{D}_{m},\mathfrak{B}_{1},\ldots,\mathfrak{B}_{r}$
be dimensions in $Q/{\sim}$ such that $\left\{ \mathfrak{B}_{1},\ldots,\mathfrak{B}_{r}\right\} $
is a maximal set of independent dimensions in the dimension group
generated by $\mathfrak{D}_{1},\ldots,\mathfrak{D}_{m},\mathfrak{B}_{1},\ldots,\mathfrak{B}_{r}$,
and let 
\[
\Psi:\mathfrak{D}_{1}\times\cdots\times\mathfrak{D}_{m}\times\mathfrak{B}_{1}\times\cdots\times\mathfrak{B}_{r}\rightarrow\mathfrak{D},\quad\left(p_{1},\dots,p_{m},q_{1},\dots,q_{r}\right)\mapsto p
\]
be a non-trivial quantity function with a covariant scalar representation
\[
\psi:\mathbb{R}_{>0}^{m+r}\rightarrow\mathbb{R}_{>0},\quad\left(s_{1},\dots,s_{m},t_{1},\dots,t_{r}\right)\mapsto s.
\]

Then there are integers $\mathsf{k}>0,\mathsf{k}_{1},\ldots,\mathsf{k}_{r}$
such that $\mathfrak{D}^{\mathsf{k}}=\prod_{j=1}^{r}\mathfrak{B}{}_{j}^{\mathsf{k}_{j}}$
and integers $\mathsf{c}_{i}>0,\mathsf{c}_{i1},\ldots,c_{ir}$ such
that $\mathfrak{D}_{i}^{\mathsf{c}_{i}}=\prod_{j=1}^{r}\mathfrak{B}_{j}^{\mathsf{c}_{ij}}$
for $i=1,\ldots,m$ such that there is
\begin{enumerate}
\item a quasiscalar function $\Phi:$ $\underset{m}{\underbrace{\left[1_{Q}\right]\times\ldots\times\left[1_{Q}\right]}}\rightarrow\left[1_{Q}\right]$
such that 
\[
p^{\mathsf{k}}=\Psi^{\mathsf{k}}\!\left(p_{1},\dots,p_{m},q_{1},\dots,q_{r}\right)=\prod_{j=1}^{r}q_{j}^{\mathsf{k}_{j}}\,\Phi\!\left(\Pi_{1},\ldots,\Pi{}_{m}\right),
\]
where $\Psi^{\mathsf{k}}\!\left(p_{1},\dots,p_{m},q_{1},\dots,q_{r}\right)=\left(\Psi\!\left(p_{1},\dots,p_{m},q_{1},\dots,q_{r}\right)\right)^{\mathsf{k}}$
\\
and $\Pi_{i}=p_{i}^{\mathsf{c}_{i}}\,/\left(\prod_{j=1}^{r}q_{j}^{\mathsf{c}_{ij}}\right)$;
\item a scalar function $\phi:\mathbb{R}_{>0}^{m}\rightarrow\mathbb{R}_{>0}$
such that $\psi^{\mathsf{k}}$, defined by 
\[
s^{\mathsf{k}}=\psi^{\mathsf{k}}\!\left(s_{1},\dots,s_{m},t_{1},\dots,t_{r}\right)=\prod_{j=1}^{r}t_{j}^{\mathsf{k}_{j}}\,\phi\!\left(\pi_{1},\ldots,\pi_{m}\right),
\]
where $\pi_{i}=s_{i}^{\mathsf{c}_{i}}\,/\left(\prod_{j=1}^{r}t_{j}^{\mathsf{c}_{ij}}\right)$,
is a covariant representation of $\Psi^{\mathsf{k}}$.
\end{enumerate}
\end{thm}
\begin{proof}
By assumption, there are integers $\mathsf{c}_{i},\mathsf{c}_{ij}$
such that $\mathfrak{D}_{i}^{\mathsf{c}_{i}}=\prod_{j=1}^{r}\mathfrak{B}_{j}^{\mathsf{c}_{ij}}$
for $i=1,\ldots,m$, and by Lemma \ref{s5.4} there is a unique function,
\[
\Psi':\mathfrak{D}_{1}^{\mathsf{c}_{1}}\times\cdots\times\mathfrak{D}_{m}^{\mathsf{c}_{m}}\times\mathfrak{B}_{1}\times\cdots\times\mathfrak{B}_{r}\rightarrow\mathfrak{D},\quad\left(p_{1}^{\mathsf{c}_{1}},\dots,p_{m}^{\mathsf{c}_{m}},q_{1},\dots,q_{r}\right)\mapsto p
\]
such that 
\[
\Psi'\!\left(p_{1}^{\mathsf{c}_{1}},\dots,p_{m}^{\mathsf{c}_{m}},q_{1},\dots,q_{r}\right)=\Psi\!\left(p_{1},\dots,p_{m},q_{1},\dots,q_{r}\right).
\]

$\Psi'$ has a covariant scalar representation by Proposition \ref{s4.7},
so by Theorem \ref{s5.1} (1), there are integers $\mathsf{k}>0,\mathsf{k}_{1},\ldots,\mathsf{k}_{r}$
such that $\mathfrak{D}^{\mathsf{k}}=\prod_{j=1}^{r}\mathfrak{B}{}_{j}^{\mathsf{k}_{j}}$.
Hence, there is a unique function 
\[
\Xi:\mathfrak{D}_{1}^{\mathsf{c}_{1}}\times\cdots\times\mathfrak{D}_{m}^{\mathsf{c}_{m}}\times\mathfrak{B}_{1}\times\cdots\times\mathfrak{B}_{r}\rightarrow\left[1_{Q}\right],\quad\left(p_{1}^{\mathsf{c}_{1}},\dots,p_{m}^{\mathsf{c}_{m}},q_{1},\dots,q_{r}\right)\mapsto\Pi
\]
such that 
\[
p^{\mathsf{k}}=\left(\Psi'\!\left(p_{1}^{\mathsf{c}_{1}},\dots,p_{m}^{\mathsf{c}_{m}},q_{1},\dots,q_{r}\right)\right)^{\mathsf{k}}=\prod_{j=1}^{r}q_{j}^{\mathsf{k}_{j}}\,\,\Xi\!\left(p_{1}^{\mathsf{c}_{1}},\dots,p_{m}^{\mathsf{c}_{m}},q_{1},\dots,q_{r}\right).
\]
By Proposition \ref{s4.5} and Proposition \ref{s4.8}, $\Xi$ has
a covariant scalar representation, since $\Psi'$ has a covariant
scalar representation.

As $\mathfrak{D}_{i}^{\mathsf{c}_{i}}=\prod_{j=1}^{r}\mathfrak{B}_{j}^{\mathsf{c}_{ij}}$
for $i=1,\ldots,m$, we can define a quantity function 
\[
\Xi':\left[1_{Q}\right]\times\ldots\left[1_{Q}\right]\times\mathfrak{B}_{1}\times\ldots\mathfrak{B}_{r}\rightarrow\left[1_{Q}\right],\quad\left(\Pi_{1},\ldots,\Pi_{m},q_{1},\ldots,q_{r}\right)\mapsto\Pi
\]
by setting 
\begin{gather*}
\Xi'\!\left(\Pi_{1},\dots,\Pi_{m},q_{1},\dots,q_{r}\right)=\\
=\Xi\!\left(\Pi_{1}\prod_{j=1}^{r}q_{j}^{\mathsf{c}_{1j}},\ldots,\Pi{}_{m}\prod_{j=1}^{r}q_{j}^{\mathsf{c}_{mj}},q_{1},\ldots,q_{r}\right)=\Xi\!\left(p_{1}^{\mathsf{c}_{1}},\ldots,p_{m}^{\mathsf{c}_{m}},q_{1},\ldots,q_{r}\right),
\end{gather*}
where $\Pi_{i}=p_{i}^{\mathsf{c}_{i}}\,/\left(\prod_{j=1}^{r}q_{j}^{\mathsf{c}_{ij}}\right)$.

$\Xi'$ has a covariant scalar representation by Proposition \ref{s4.7},
so by Theorem \ref{s5.1} (2), there is a function 
\[
\Phi:\left[1_{Q}\right]\times\ldots\times\left[1_{Q}\right]\rightarrow\left[1_{Q}\right],\qquad\left(\Pi_{1},\ldots,\Pi{}_{m}\right)\mapsto\Pi
\]
such that $\Phi\!\left(\Pi_{1},\ldots,\Pi{}_{m}\right)=\Xi'\!\left(\Pi_{1},\dots,\Pi_{m},q_{1},\dots,q_{r}\right)$.
This means that\linebreak{}
 $\Xi\!\left(p_{1}^{\mathsf{c}_{1}},\dots,p_{m}^{\mathsf{c}_{m}},q_{1},\dots,q_{r}\right)=\Phi\!\left(\Pi_{1},\ldots,\Pi{}_{m}\right)$,
proving (1).

(2) follows easily from Propositions \ref{s4.2}, \ref{s4.5} and
\ref{s4.8}.$\qedhere$ 
\end{proof}
Without loss of generality, one can require that the integers $\mathsf{k},\mathsf{k}_{1},\ldots,\mathsf{k}_{r}$
are relatively prime, and similarly that the integers $\mathsf{c}_{i},\mathsf{c}_{i1},\ldots,\mathsf{c}_{ir}$
are relatively prime for $i=1,\ldots,m$. 
\begin{cor}
\label{s5.5}(Conventional $\Pi$ theorem.) Under the same assumptions
as in Theorem \ref{s5.3}, there are integers $\mathsf{k}>0,\mathsf{k}_{j}$
and $\mathsf{c}_{i}>0,\mathsf{c}_{ij}$ for $i=1,\ldots,m$ and $j=1,\ldots,r$
such that there is
\begin{enumerate}
\item a quasiscalar function $\varPhi:$$\underset{m}{\underbrace{\left[1_{Q}\right]\times\ldots\times\left[1_{Q}\right]}}\rightarrow\left[1_{Q}\right]$
such that 
\[
p=\Psi\!\left(p_{1},\dots,p_{m},q_{1},\dots,q_{r}\right)=\sqrt[\mathsf{k}]{\prod_{j=1}^{r}q_{j}^{\mathsf{k}_{j}}}\,\varPhi\!\left(\Pi{}_{1},\ldots,\Pi_{m}\right),
\]
where $\Pi_{i}=p_{i}^{\mathsf{c}_{i}}\,/\left(\prod_{j=1}^{r}q_{j}^{\mathsf{c}_{ij}}\right)$;
\item a scalar function $\varphi:\mathbb{R}_{>0}^{m}\rightarrow\mathbb{R}_{>0}$
such that $\psi$, defined by 
\[
s=\psi\!\left(s_{1},\dots,s_{m},t_{1},\dots,t_{r}\right)=\prod_{j=1}^{r}t_{j}^{\mathsf{k}_{j}/\mathsf{k}}\,\varphi\!\left(\pi_{1},\ldots,\pi_{m}\right),
\]
where $\pi_{i}=s_{i}^{\mathsf{c}_{i}}\,/\left(\prod_{j=1}^{r}t_{j}^{\mathsf{c}_{ij}}\right)$,
is a covariant representation of $\Psi$. 
\end{enumerate}
\end{cor}
\begin{proof}
This follows immediately from Theorem \ref{s5.3}, using Lemma \ref{s5.4}.$\qedhere$
\end{proof}
One can also replace scalar ratios of the form $s_{i}^{\mathsf{c}_{i}}\,/\left(\prod_{j=1}^{r}t_{j}^{\mathsf{c}_{ij}}\right)$
in the representation theorems with ratios of the form $s_{i}\,/\left(\prod_{j=1}^{r}t_{j}^{\mathsf{c}_{ij}/\mathsf{c_{i}}}\right)$,
so that, for example, 
\[
s=\prod_{j=1}^{r}t_{j}^{\mathsf{k}_{j}/\mathsf{k}}\,\mathcal{F}\!\left(s_{1}\,/\left(\prod_{j=1}^{r}t_{j}^{\mathsf{c}_{1j}/\mathsf{c}_{1}}\right),\ldots,s_{m}\,/\left(\prod_{j=1}^{r}t_{j}^{\mathsf{c}_{mj}/\mathsf{c}_{m}}\right)\right).
\]
Using Lemma \ref{s5.4}, one can similarly replace corresponding quantity
ratios with roots of quantity ratios.

\subsection{Variations on the representation theorems}

The representation theorems as stated above apply only to quantity
spaces over a scalar system $\mathcal{R}$ where all scalars are positive.
In physics and engineering, measures of quantities are usually non-negative,
and it is often natural to consider only positive measures of important
types of quantities such as distance, time elapsed, mass, and (absolute)
temperature. Sometimes, the assumption that all scalars in scalar
representations are positive is not appropriate, however, so the representation
theorems need to be modified.

An analysis of the proofs of the representation theorems reveals that
the assumption about positive scalars is used in two ways. First,
the assumption guarantees that $q_{1},\ldots,q_{n}$ are invertible.
Hence, the quantities $p_{i}^{\mathsf{c}_{i}}\,/\left(\prod_{j=1}^{r}q_{j}^{\mathsf{c}_{ij}}\right)$
and the scalars $s_{i}^{\mathsf{c}_{i}}\,/\left(\prod_{j=1}^{r}t_{j}^{\mathsf{c}_{ij}}\right)$
are defined. Instead, we could assume that $\Phi$ is a partial quantity
function, defined only for invertible quantities $\prod_{j=1}^{r}q_{j}^{\mathsf{c}_{ij}}$.
Second, positive scalars guarantee that unique roots of quantities
are defined as described above (though actually only non-negative
scalars are required for unique roots to exist). However, if the integers
$\mathsf{k},\mathsf{c}_{i}$ in the representation theorems are all
equal to $1$ then unique 'roots' exist for all quantities, so one
could use this assumption instead of the assumption about positive
scalars. To clarify this, let us consider an example.

Let $\mathfrak{D}$ be a dimension in $Q/{\sim}$, where $Q$ is a
quantity space over a scalar system of \emph{non-negative} real numbers.
Consider the quantity function 
\[
\Psi:\mathfrak{D}\times\mathfrak{D\rightarrow\mathfrak{D}},\qquad\left(q_{a},q_{b}\right)\mapsto q_{a}+q_{b}.
\]
Since $\mu_{B}\!\left(q\right)+\mu_{B}\!\left(q'\right)=\mu_{B}\!\left(q+q'\right)$
for any $B$, $\Psi$ has a covariant scalar representation 
\[
\psi:\mathbb{R}_{\geq0}\times\mathbb{R}_{\geq0}\rightarrow\mathbb{R}_{\geq0},\qquad\left(a,b\right)\mapsto a+b.
\]
$\psi$ does not have the form stipulated by the representation theorems.
Consider, however, the closely related function 
\[
\psi':\mathbb{R}_{>0}\times\mathbb{R}_{>0}\rightarrow\mathbb{R}_{>0},\qquad\left(a,b\right)\mapsto a\,\phi\!\left(b/a\right),
\]
where $\phi\!\left(x\right)=1+x$. If $a>0$ then $\psi'\!\left(a,b\right)=a\left(1+b/a\right)=a+b$.
Thus, $\psi$ is equal to $\psi^{*}$ defined by 
\[
\begin{cases}
\psi^{*}\!\left(a,b\right)=\psi'\!\left(a,b\right) & \left(a>0,b>0\right)\\
\psi^{*}\!\left(a,b\right)=a & \left(a>0,b=0\right)\\
\psi^{*}\!\left(a,b\right)=b & \left(a=0,b>0\right)\\
\psi^{*}\!\left(a,b\right)=0 & \left(a=0,b=0\right)
\end{cases}.
\]
It is always possible to 'patch' the scalar representation where it
is not defined, but it is not clear that this can always be done in
a natural way. In this case $\psi'$ can be extended in a natural
way, since $\underset{a\rightarrow0}{\lim}\, a\left(1+b/a\right)=b$.

It should also be noted that we can represent the relation $c=a+b$
in the form $c=a\left(1+b/a\right)$ in cases where $a$ and $b$
may be \emph{negative} real numbers, too. This illustrates the fact
that we can in some cases dispense with an explicit or implicit assumption
underlying the usual $\Pi$ theorem -- that scalars are positive (or
possibly non-negative) numbers.

\section{A matrix interpretation of the representation theorems}

Let $B=\left\{ B_{1},\ldots,B_{n}\right\} $ be a basis for $Q/{\sim}$,
and let $\left[q_{0}\right],\ldots,\left[q_{m}\right]$ be dimensions
in $Q/{\sim}$. We can represent the expansions of $\left[q_{0}\right],\ldots,\left[q_{m}\right]$
in terms of $B$ by a \emph{dimensional matrix 
\[
\begin{array}{c}
\begin{array}{c}
\vphantom{D_{1}}\\
B_{1}\\
\vdots\\
B_{j}\\
\vdots\\
B_{n}
\end{array}\left\Vert \begin{array}{ccccc}
\left[q_{0}\right] & \cdots & \left[q_{i}\right] & \cdots & \left[q_{m}\right]\\
a_{01} & \cdots & a_{i1} & \cdots & a_{m1}\\
\vdots &  & \vdots &  & \vdots\\
a_{0j} & \cdots & a_{ij} & \cdots & a_{mj}\\
\vdots &  & \vdots &  & \vdots\\
a_{0n} & \cdots & a_{in} & \cdots & a_{mn}
\end{array}\right\Vert \end{array}
\]
}such that 
\[
\left[q_{i}\right]=\prod_{j=1}^{n}B_{j}^{a_{ij}}\qquad\left(i=0,\ldots,m\right).
\]

Let $q_{i}^{\nu}$ denote the column vector $\left[a_{i1}\cdots a_{in}\right]^{\mathrm{T}}$
corresponding to $\left[q_{i}\right]$, let $\left[q\right]$ be one
of the dimensions $\left[q_{0}\right],\ldots,\left[q_{m}\right]$,
and let $q^{\nu}$ be the corresponding column vector. \linebreak{}
$\left[q\right]$ is dependent on $\left[q_{i_{1}}\right],\ldots,\left[q_{i_{r}}\right]$,
meaning that there are integers $k>0,k_{i_{j}}$ such that 
\[
\left[q\right]^{k}=\prod_{j=1}^{r}\left[q_{i_{j}}\right]^{k_{i_{j}}},
\]
if and only if 
\[
kq^{\nu}=\sum_{j=1}^{r}k_{i_{j}}q_{i_{j}}^{\nu}.
\]

It is also clear that any dimensions $\left[q_{i_{1}}\right],\ldots,\left[q_{i_{r}}\right]$
are independent if and only if $q_{i_{1}}^{\nu},\ldots,q_{i_{r}}^{\nu}$
are independent as column vectors, meaning that if 
\[
\sum_{j=1}^{r}k_{i_{j}}q_{i_{j}}^{\nu}=\left[\begin{array}{c}
0\\
\cdots\\
0
\end{array}\right]
\]
then $k_{i_{j}}=0$ for $j=1,\ldots,r$.

The facts just mentioned have important immediate consequences:
\begin{thm}
\label{s6.1}Let $D$ be a dimensional matrix with column head dimensions\linebreak{}
 $\left[q_{0}\right],\ldots,\left[q_{m}\right]$. Then the column
vector $q^{\nu}$ in $D$ is dependent on $q_{i_{1}}^{\nu},\ldots,q_{i_{r}}^{\nu}$
if and only if $\left[q\right]$ is dependent on $\left[q_{i_{1}}\right],\ldots,\left[q_{i_{r}}\right]$,
and $\left\{ q_{i_{1}}^{\nu},\ldots,q_{i_{r}}^{\nu}\right\} $ is
a (maximal) set of independent column vectors in $D$ if and only
if $\left\{ \left[q_{i_{1}}\right],\ldots,\left[q_{i_{r}}\right]\right\} $
is a (maximal) set of independent dimensions in $\left\{ \left[q_{0}\right],\ldots,\left[q_{m}\right]\right\} $. \end{thm}
\begin{cor}
\label{s6.2}The rank of a dimensional matrix is the same as the rank
of the dimension group generated by the column head dimensions $\left[q_{0}\right],\ldots,\left[q_{m}\right]$. 
\end{cor}
To be able to apply the representation theorems, we need to prepare
the data in the dimensional matrix. First, we select a dependent dimension,
denoted $\mathfrak{D}$ in the representation theorems. Then, the
other dimensions should be divided into a non-empty set $I$ of independent
dimensions and a possibly empty set $D$ of dimensions dependent on
those in $I$. These are the dimensions denoted $\mathfrak{B}_{1},\ldots,\mathfrak{B}_{r}$
and $\mathfrak{D}_{1},\ldots,\mathfrak{D}_{m}$, respectively, in
the representation theorems. It is also required that $\mathfrak{D}$
is dependent on the dimensions in $I$. It is clear that $I$ is a
\emph{maximal set of independent dimensions not containing $\mathfrak{D}$}
in the given set of column head dimensions. While neither the existence
nor the uniqueness of such a set $I$ is guaranteed, there is for
each dimensional matrix a set $\mathscr{M}$ of all maximal sets of
independent dimensions not containing\emph{ $\mathfrak{D}$} in the
given set of column dimensions, and a corresponding set of sets of
column vectors. For each set in $\mathscr{M}$ there is a unique \emph{dimensional
model} to which the representation theorems can be applied.

For example, consider the following three dimensional matrices: 
\[
\begin{array}{c}
\\
B_{1}\\
B_{2}
\end{array}\left\Vert \begin{array}{cc}
\left[q\right] & \left[q_{1}\right]\\
1 & 1\\
0 & 1
\end{array}\right\Vert ,\quad\;\begin{array}{c}
\\
B_{1}\\
B_{2}
\end{array}\left\Vert \begin{array}{ccc}
\left[q\right] & \left[q_{1}\right] & \left[q_{2}\right]\\
1 & 1 & 2\\
0 & 1 & 0
\end{array}\right\Vert ,\quad\;\begin{array}{c}
\\
B_{1}\\
B_{2}
\end{array}\left\Vert \begin{array}{cccc}
\left[q\right] & \left[q_{1}\right] & \left[q_{2}\right] & \left[q_{3}\right]\\
1 & 1 & 2 & 0\\
0 & 1 & 0 & 1
\end{array}\right\Vert .
\]

The rank of each matrix is 2, so every maximal set of independent
column vectors contains 2 column vectors. Equivalently, every maximal
set of independent column head dimensions contains 2 dimensions. We
let $\left[q\right]$ be the dependent dimension, so that $q$ is
the dependent variable.

In the first case, there is exactly one maximal set of independent
dimensions, namely $\left\{ \left[q\right],\left[q_{1}\right]\right\} $,
but this set contains $\left[q\right]$, so $\mathscr{M}$ is empty
and the representation theorems do not apply.

In the second case, however, there is exactly one maximal set of independent
dimensions which does not contain $\left[q\right]$, namely $\left\{ \left[q_{1}\right],\left[q_{2}\right]\right\} $,
so the dimensional matrix translates into a unique dimensional model:
\[
\begin{array}{c}
\left\Vert \begin{array}{c}
\left(\mathfrak{D}\right)\\
\left[q\right]\\
1\\
0
\end{array}\right|\left|\begin{array}{cc}
\left(\mathfrak{B}_{1}\right) & \left(\mathfrak{B}_{2}\right)\\
\left[q_{1}\right] & \left[q_{2}\right]\\
1 & 2\\
1 & 0
\end{array}\right\Vert \end{array}.
\]
According to the representation theorems, dimensional analysis involves
solving the equation 
\[
\left[q\right]^{k}=\left[q_{1}\right]^{k_{1}}\left[q_{2}\right]^{k_{2}},
\]
or equivalently the linear \emph{dimensional equation} 
\[
kq^{\nu}=k_{1}q_{1}^{\nu}+k_{2}q_{2}^{\nu},
\]
that is, 
\[
k\left[\begin{array}{c}
1\\
0
\end{array}\right]=k_{1}\left[\begin{array}{c}
1\\
1
\end{array}\right]+k_{2}\left[\begin{array}{c}
2\\
0
\end{array}\right],
\]
where $k,k_{1},k_{2}$ are integers and $k>0$.

Note that the matrix of the homogeneous equation system corresponding
to the dimensional equation has $r+1$ columns, $r$ of which are
independent. Thus, the solution space is one-dimensional; all solutions
of a dimensional equation have the form $\lambda\left(k,k_{1},\ldots,k_{r}\right)$,
where $\left(k,k_{1},\ldots,k_{r}\right)$ is a basis for the solution
space. Furthermore, $k\neq0$ (and we may assume that $k>0$) since
$q^{\nu}$ is dependent on $q_{1}^{\nu},\ldots,q_{r}^{\nu}$, and
there is a basis vector for the solution space where all entries are
integers, since all coefficients in the equation system considered
are integers and a basis vector for the solution space can be derived
from the matrix representing a dimensional equation by means of integer-preserving
operations \cite{key-1}.

Thus, it is clear that there is a unique integer solution $\left(k,k_{1},\ldots,k_{r}\right)$,
where $k>0$ and $k\leq k'$ for any integer solution $\left(k',k_{1}',\ldots,k_{r}'\right)$;
it suffices to consider only such \emph{canonical solutions}, since
all solutions with integer coefficients have the form $n\left(k,k_{1},\ldots,k_{r}\right)$,
where $n$ is an integer and $\left(k,k_{1},\ldots,k_{r}\right)$
is the canonical solution.

The canonical solution in the case considered is 
\[
k=2,\, k_{1}=0,\, k_{2}=1,
\]
with the corresponding quantity relation 
\[
q^{2}=q_{2}\,\Phi\!\left(\right)=q_{2}\left(K1_{Q}\right)=Kq_{2}\qquad\mathrm{or}\qquad q=C\sqrt{q_{2}},
\]
and the corresponding scalar relation $t^{2}=Kt_{2}$ or $t=C\sqrt{t_{2}}$.

In the third case, $\mathscr{M}=\left\{ \left\{ \left[q_{2}\right],\left[q_{3}\right]\right\} ,\left\{ \left[q_{1}\right],\left[q_{3}\right]\right\} ,\left\{ \left[q_{1}\right],\left[q_{2}\right]\right\} \right\} $,
and we have three different dimensional models, as shown below: 
\[
\begin{array}{c}
\left\Vert \begin{array}{c}
\left(\mathfrak{D}\right)\\
\left[q\right]\\
1\\
0
\end{array}\right|\begin{array}{c}
\left(\mathfrak{D}_{1}\right)\\
\left[q_{1}\right]\\
1\\
1
\end{array}\left|\begin{array}{cc}
\left(\mathfrak{B}_{1}\right) & \left(\mathfrak{B}_{2}\right)\\
\left[q_{2}\right] & \left[q_{3}\right]\\
2 & 0\\
0 & 1
\end{array}\right\Vert ,\end{array}\qquad\begin{array}{c}
\left\Vert \begin{array}{c}
\left(\mathfrak{D}\right)\\
\left[q\right]\\
1\\
0
\end{array}\right|\begin{array}{c}
\left(\mathfrak{D}_{1}\right)\\
\left[q_{2}\right]\\
2\\
0
\end{array}\left|\begin{array}{cc}
\left(\mathfrak{B}_{1}\right) & \left(\mathfrak{B}_{2}\right)\\
\left[q_{1}\right] & \left[q_{3}\right]\\
1 & 0\\
1 & 1
\end{array}\right\Vert ,\end{array}
\]
\begin{gather*}
\left\Vert \begin{array}{c}
\left(\mathfrak{D}\right)\\
\left[q\right]\\
1\\
0
\end{array}\right|\begin{array}{c}
\left(\mathfrak{D}_{1}\right)\\
\left[q_{3}\right]\\
0\\
1
\end{array}\left|\begin{array}{cc}
\left(\mathfrak{B}_{1}\right) & \left(\mathfrak{B}_{2}\right)\\
\left[q_{1}\right] & \left[q_{2}\right]\\
1 & 2\\
1 & 0
\end{array}\right\Vert .
\end{gather*}
For the first model, we have the equations 
\[
k\left[\begin{array}{c}
1\\
0
\end{array}\right]=k_{2}\left[\begin{array}{c}
2\\
0
\end{array}\right]+k_{3}\left[\begin{array}{c}
0\\
1
\end{array}\right],\qquad k_{1}\left[\begin{array}{c}
1\\
1
\end{array}\right]=k_{2}'\left[\begin{array}{c}
2\\
0
\end{array}\right]+k_{3}'\left[\begin{array}{c}
0\\
1
\end{array}\right]
\]
with the canonical solutions 
\[
\begin{cases}
k=2,k_{2}=1,k_{3}=0\\
k_{1}=2,k_{2}'=1,k_{3}'=2
\end{cases}
\]
and the corresponding quantity relation 
\[
q^{2}=q_{2}\,\Phi{}_{1}\!\!\left(\frac{q_{1}^{2}}{q_{2}\, q_{3}^{2}}\right).
\]
Quantity relations corresponding to the two other dimensional models
are 
\[
q=q_{1}q_{3}^{-1}\,\Phi{}_{2}\!\left(\frac{q_{2}}{q_{1}^{2}q_{3}^{-2}}\right),\qquad q^{2}=q_{2}\,\Phi{}_{3}\!\left(\frac{q_{3}^{2}}{q_{1}^{2}q_{2}^{-1}}\right).
\]

The fact that there is in general more than one dimensional model
corresponding to a given dimensional matrix, even if the dependent
dimension is fixed, may at first seem to be an unwanted complication,
casting doubt on the capacity of dimensional analysis to produce unambiguous
results. In applications of dimensional analysis, this complication
is usually hidden by quietly choosing one quantity relation -- or
mostly the corresponding scalar relation -- in cases where several
relations can be derived from the dimensional matrix. On reflection
it is clear, however, that all relations derived are valid, and that
these relations combined may yield more information than any single
one. This means that by considering all dimensional models we can
often gain more information about the relation between the quantities
involved than by considering just one dimensional model. Examples
illustrating this principle will be given in the next section.

\section{Dimensional analysis exemplified}
\begin{example}
We want to express the area $a$ of a circle as a function of its
diameter $d$. Assume that $\left\{ L\right\} $ is a basis for the
dimension group considered, and that $\left[a\right]=L{}^{2}$ and
$\left[d\right]=L$. Intuitively, $L$ is the dimension of length.
The dimensional matrix is 
\[
\begin{array}{c}
\begin{array}{c}
\\
L
\end{array}\left\Vert \begin{array}{cc}
\left[a\right] & \left[d\right]\\
2 & 1
\end{array}\right\Vert \end{array},
\]
and $\left\{ \left[d\right]\right\} $ is the only maximal set of
independent dimensions not containing $\left[a\right]$, so the dimensional
model for the relation $a=\Psi\!\left(d\right)$ is simply 
\[
\begin{array}{cc}
\left\Vert \begin{array}{c}
\left[a\right]\\
2
\end{array}\right|\! & \!\left|\begin{array}{c}
\left[d\right]\\
1
\end{array}\right\Vert \end{array}.
\]
This gives a dimensional equation $k_{a}a^{\nu}=k_{d}d^{\nu}$ with
the canonical solution \linebreak{}
 $\left\{ k_{a}=1,\; k_{d}=2\right\} $ and hence we obtain the quantity
relation 
\[
a=d^{2}\Phi\!\left(\right)=Kd^{2},
\]
where $\Phi\!\left(\right)=K1_{Q}\in\left[1_{Q}\right]$, or the formally
almost identical scalar relation 
\[
a=d^{2}\phi\!\left(\right)=Kd^{2},
\]
where $\phi\!\left(\right)=K\in\mathbb{R}{}_{>0}$. The constant $K$
cannot be determined by means of dimensional analysis, but we know
from elementary geometry that $K=\pi/4$. 
\end{example}
\begin{example}
The following example of dimensional analysis is often given. Assume
that the time of oscillation $t$ of a pendulum depends on its length
$l$, the mass of the bob $m$, the amplitude of the oscillation (in
radians) $\theta$ and the constant of gravity $g$; that is, $t=\Psi\!\left(l,m,\theta,g\right)$.
We obtain the following dimensional matrix 
\[
\begin{array}{r}
\begin{array}{c}
\\
L\\
T\\
M
\end{array}\left\Vert \begin{array}{ccccc}
\left[t\right] & \left[l\right] & \left[m\right] & \left[\theta\right] & \left[g\right]\\
0 & 1 & 0 & 0 & 1\\
1 & 0 & 0 & 0 & -2\\
0 & 0 & 1 & 0 & 0
\end{array}\right\Vert \end{array},
\]
and there is only one maximal set of independent dimensions not containing
the dimension $\left[t\right]$ corresponding to the dependent variable
$t$, namely $\left\{ \left[l\right],\left[m\right],\left[g\right]\right\} $.
The corresponding dimensional model is 
\[
\left\Vert \begin{array}{c}
\left[t\right]\\
0\\
1\\
0
\end{array}\right|\begin{array}{c}
\left[\theta\right]\\
0\\
0\\
0
\end{array}\left|\begin{array}{ccc}
\left[l\right] & \left[m\right] & \left[g\right]\\
1 & 0 & 1\\
0 & 0 & -2\\
0 & 1 & 0
\end{array}\right\Vert .
\]
This model gives two dimensional equations 
\[
\begin{cases}
k_{t}t^{\nu}=k_{l}l^{\nu}+k_{m}m^{\nu}+k_{g}g^{\nu}\\
k_{\theta}\theta^{\nu}=k_{l}'l^{\nu}+k_{m}'m^{\nu}+k_{g}'g^{\nu}
\end{cases}
\]
with canonical solutions 
\[
\begin{cases}
k_{t}=2,\; k_{l}=1,\; k_{m}=0,\; k_{g}=-1\\
k_{\theta}=1,\; k_{l}'=0,\; k_{m}'=0,\; k_{g}'=0
\end{cases},
\]
and so we obtain a quantity relation 
\[
t^{2}=\left(l/g\right)\,\Phi\!\left(\theta/1_{Q}\right)=\left(l/g\right)\,\Phi\!\left(\theta\right),
\]
where $\Phi:\left[1_{Q}\right]\rightarrow\left[1_{Q}\right]$, or
the equivalent scalar relation 
\[
t^{2}=\left(l/g\right)\,\phi\!\left(\theta\right),
\]
where $\phi:\mathbb{R}_{>0}\rightarrow\mathbb{R}_{>0}$. Thus, the
time of oscillation does not really depend on the mass of the bob.%
\footnote{The assumption that $t=\Psi\!\left(l,m,\theta,g\right)$ means, to
be precise, that $t$ does \emph{not} depend on \emph{other} variables
than $l$, $m$, $\theta$ and $g$, so this fact does not contradict
the original assumption. Assumptions about governing variables in
dimensional analysis are always of this nature.%
} It is known that for $\theta\approx0$ we have $ $$\phi\!\left(\theta\right)\approx4\pi^{2}$,
so for small oscillations both the quantity relation and the scalar
relation can be written as 
\[
t^{2}=4\pi^{2}\left(l/g\right)\quad\mathrm{or}\quad t=2\pi\sqrt{l/g}\,;
\]
note that the quantity square root is well-defined because $\left[l/g\right]=\left[t\right]^{2}$. 
\end{example}
For simplicity, we shall be using symbols such as $\Psi,\Psi_{1},\Psi_{2}$
and $\Phi,\Phi{}_{1},\Phi{}_{2}$ to denote both quantity functions
and scalar functions in the examples below, and each relation derived
can be interpreted both as a quantity relation and a scalar relation,
unless otherwise indicated.
\begin{example}
We want to express the area $a$ of a rectangle as a function of the
length of its longest side $l$ and the length of its shortest side
$s$; that is, $a=\Psi\!\left(l,s\right)$. The dimensional matrix
is 
\[
\begin{array}{c}
\\
L
\end{array}\left\Vert \begin{array}{ccc}
\left[a\right] & \left[l\right] & \left[s\right]\\
2 & 1 & 1
\end{array}\right\Vert ,
\]
and there are two corresponding dimensional models with $a$ as the
dependent variable, 
\[
\begin{array}{c}
\left\Vert \begin{array}{c}
\left[a\right]\\
2
\end{array}\right|\begin{array}{c}
\left[l\right]\\
1
\end{array}\left|\begin{array}{c}
\left[s\right]\\
1
\end{array}\right\Vert \end{array},\qquad\begin{array}{c}
\begin{array}{c}
\left\Vert \begin{array}{c}
\left[a\right]\\
2
\end{array}\right|\begin{array}{c}
\left[s\right]\\
1
\end{array}\left|\begin{array}{c}
\left[l\right]\\
1
\end{array}\right\Vert .\end{array}\end{array}
\]
The two corresponding sets of dimensional equations are 
\[
\begin{cases}
k_{a}a^{\nu}=k_{s}s^{\nu}\\
k_{l}l^{\nu}=k_{s}'s^{\nu}
\end{cases}\quad\begin{cases}
k_{a}a^{\nu}=k_{l}l^{\nu}\\
k_{s}s^{\nu}=k_{l}'l^{\nu}
\end{cases}
\]
with the canonical solutions 
\[
\begin{cases}
k_{a}=1,\; k_{s}=2\\
k_{l}=1,\; k_{s}'=1
\end{cases}\quad\begin{cases}
k_{a}=1,\; k_{l}=2\\
k_{s}=1,\; k_{l}'=1
\end{cases}\!\!\!.
\]
These solutions give the relations 
\[
a=s^{2}\Phi_{1}\!\left(l/s\right),\quad a=l^{2}\Phi{}_{2}\!\left(s/l\right).
\]
If we now assume, for symmetry reasons, that $\Phi{}_{1}=\Phi_{2}=\Phi$,
we see that $s^{2}\Phi\!\left(l/s\right)=l^{2}\Phi\!\left(s/l\right)$,
and setting $x=l/s$ this becomes 
\[
\Phi\!\left(x\right)=x^{2}\Phi\!\left(1/x\right).
\]
This functional equation has solutions of the form $\Phi\!\left(x\right)=Kx.$
Hence, from\linebreak{}
 $a=s^{2}\Phi\!\left(l/s\right)=l^{2}\Phi\!\left(s/l\right)$ we finally
obtain the relations 
\[
a=Ksl=Kls
\]
for some $K$, and we know that $K=1$ for a rectangle. 
\end{example}
\begin{example}
Now, we express the area $a$ of an ellipse as a function $a=\Psi\!\left(t,c\right)$
of its transverse diameter $t$ and its conjugate diameter $c$. The
dimensional matrix for this problem has exactly the same form as in
the previous example; it is 
\[
\begin{array}{c}
\\
L
\end{array}\left\Vert \begin{array}{ccc}
\left[a\right] & \left[t\right] & \left[c\right]\\
2 & 1 & 1
\end{array}\right\Vert .
\]
By the same reasoning, in particular the same symmetry assumption
as in the previous example, we obtain the relation 
\[
a=Ktc=Kct,
\]
but this time $K=\pi/4$ -- indeed, $K$ is the same for any ellipse,
and the value of $K$ for a circle is $\pi/4$.
\end{example}
\begin{example}
Let $A$ and $B$ be two bodies with mass $a$ and $b$, respectively,
and let $c$ be the combined mass of $A$ and $B$. We are looking
for a quantity function $\Psi$ such that $c=\Psi\left(a,b\right)$.
The simple dimensional matrix is 
\[
\begin{array}{c}
\\
M
\end{array}\left\Vert \begin{array}{ccc}
\left[c\right] & \left[a\right] & \left[b\right]\\
1 & 1 & 1
\end{array}\right\Vert ,
\]
and there are two corresponding dimensional models with c as the dependent
variable, 
\[
\begin{array}{c}
\left\Vert \begin{array}{c}
\left[c\right]\\
1
\end{array}\right|\begin{array}{c}
\left[a\right]\\
1
\end{array}\left|\begin{array}{c}
\left[b\right]\\
1
\end{array}\right\Vert ,\end{array}\qquad\begin{array}{c}
\left\Vert \begin{array}{c}
\left[c\right]\\
1
\end{array}\right|\begin{array}{c}
\left[b\right]\\
1
\end{array}\left|\begin{array}{c}
\left[a\right]\\
1
\end{array}\right\Vert \end{array}.
\]

These models give the relations 
\[
c=b\,\Phi_{1}\!\left(a/b\right),\quad c=a\,\Phi_{2}\!\left(b/a\right),
\]
and assuming for symmetry reasons that $\Phi{}_{1}=\Phi{}_{2}=\Phi$
and setting $x=a/b$, we obtain the functional equation 
\[
\Phi\!\left(x\right)=x\,\Phi\!\left(1/x\right),
\]
which has 
\[
\Phi\!\left(x\right)=x+1
\]
as a solution. From either one of the relations $c=b\,\Phi{}_{1}\!\left(a/b\right)$
or $c=a\,\Phi{}_{2}\!\left(b/a\right)$ we obtain 
\[
c=a+b
\]
as expected. This example illustrates the fact that although any quantity
function obtained by dimensional analysis can be defined using only
multiplication and division of quantities, addition and subtraction
can sometimes also be used to express the quantity function and hence
also its scalar representation. 
\end{example}
\begin{example}
(Adapted from Bridgman \cite{key-3}, pp. 5--8.) Let two bodies with
mass $m_{1}$ and $m_{2}$ revolve around each other in circular orbits
under influence of their mutual gravitational attraction. Let $d$
denote their distance and $t$ the time of revolution. We want to
find how $t$ depends on other quantities.

We first assume that $t=\Psi\!\left(d,m_{1},m_{2}\right)$ and formulate
a model where $L$ (length), $T$ (time) and $M$ (mass) are the elements
of a basis for the dimension group generated by $\left\{ \left[t\right],\left[d\right],\left[m_{1}\right],\left[m_{2}\right]\right\} $.
We obtain the dimensional matrix below: 
\[
\begin{array}{c}
\\
L\\
T\\
M
\end{array}\left\Vert \begin{array}{cccc}
\left[t\right] & \left[d\right] & \left[m_{1}\right] & \left[m_{2}\right]\\
0 & 1 & 0 & 0\\
1 & 0 & 0 & 0\\
0 & 0 & 1 & 1
\end{array}\right\Vert .
\]
Inspecting this matrix, we note that each subset of $\left\{ \left[t\right],\left[d\right],\left[m_{1}\right],\left[m_{2}\right]\right\} $
which is a maximal set of independent dimensions contains $\left[t\right]$.
(Specifically, $\left\{ \left[t\right],\left[d\right],\left[m_{1}\right]\right\} $
and $\left\{ \left[t\right],\left[d\right],\left[m_{2}\right]\right\} $
contain $\left[t\right]$, while $\left\{ \left[d\right],\left[m_{1}\right],\left[m_{2}\right]\right\} $
are not independent, and sets of less than 3 dimensions can be disregarded,
since the rank of the dimensional matrix is 3.) Hence, $\Psi$ cannot
have a covariant scalar representation, and we cannot use the representation
theorems to find this quantity function.

Thus, we have to add another governing variable, and Bridgman suggests
that $t$ does also depend on the gravitation constant $G$, giving
the dimensional matrix 
\[
\begin{array}{c}
\\
L\\
T\\
M
\end{array}\left\Vert \begin{array}{ccccc}
\left[t\right] & \left[d\right] & \left[m_{1}\right] & \left[m_{2}\right] & \left[G\right]\\
0 & 1 & 0 & 0 & \begin{array}{r}
3\end{array}\\
1 & 0 & 0 & 0 & -2\\
0 & 0 & 1 & 1 & -1
\end{array}\right\Vert .
\]
For this dimensional matrix, there are two possible dimensional models,
\[
\begin{array}{c}
\left\Vert \begin{array}{c}
\left[t\right]\\
1\\
0\\
0
\end{array}\right|\begin{array}{c}
\left[m_{1}\right]\\
0\\
0\\
1
\end{array}\left|\begin{array}{ccc}
\left[m_{2}\right] & \left[d\right] & \left[G\right]\\
0 & 0 & -2\\
0 & 1 & 3\\
1 & 0 & -1
\end{array}\right\Vert \end{array},\qquad\begin{array}{c}
\left\Vert \begin{array}{c}
\left[t\right]\\
1\\
0\\
0
\end{array}\right|\begin{array}{c}
\left[m_{2}\right]\\
0\\
0\\
1
\end{array}\left|\begin{array}{ccc}
\left[m_{1}\right] & \left[d\right] & \left[G\right]\\
0 & 0 & -2\\
0 & 1 & 3\\
1 & 0 & -1
\end{array}\right\Vert \end{array}.
\]

\pagebreak{}

The corresponding two sets of dimensional equations are 
\[
\begin{cases}
k_{t}t^{\nu}=k_{m_{2}}m_{2}^{\nu}+k_{d}d^{\nu}+k_{G}G^{\nu}\\
k_{m_{1}}m_{1}^{\nu}=k_{m_{2}}'m_{2}^{\nu}+k_{d}'d^{\nu}+k_{G}'G^{\nu}
\end{cases}\begin{cases}
k_{t}t^{\nu}=k_{m_{1}}m_{1}^{\nu}+k_{d}d^{\nu}+k_{G}G^{\nu}\\
k_{m_{2}}m_{2}^{\nu}=k_{m_{1}}'m_{1}^{\nu}+k_{d}'d^{\nu}+k_{G}'G^{\nu}
\end{cases}\!\!\!.
\]
The canonical solutions are 
\[
\begin{cases}
k_{t}=2,\; k_{m_{2}}=-1,\; k_{d}=3,\; k_{G}=-1\\
k_{m_{1}}=1,\; k_{m_{2}}'=1,\; k_{d}'=0,\; k_{G}'=0
\end{cases}\begin{cases}
k_{t}=2,\; k_{m_{1}}=-1,\; k_{d}=3,\; k_{G}=-1\\
k_{m_{2}}=1,\; k_{m_{1}}'=1,\; k_{d}'=0,\; k_{G}'=0
\end{cases}\!\!\!,
\]
and thus 
\[
t^{2}=\frac{d^{3}}{Gm_{2}}\:\Phi_{1\!}\!\left(\frac{m_{1}}{m_{2}}\right)=\frac{d^{3}}{Gm_{1}}\:\Phi_{2\!}\!\left(\frac{m_{2}}{m_{1}}\right).
\]

For symmetry reasons, we assume that $\Phi{}_{1}=\Phi_{2}=\Phi$,
so setting $x=\frac{m_{1}}{m_{2}}$ we obtain the functional equation
\[
x\,\Phi\!\left(x\right)=\Phi\!\left(\frac{1}{x}\right).
\]
Since $\frac{x}{1+x}=\frac{1}{1+\left(1/x\right)}$, this functional
equation has solutions of the form 
\[
\Phi\!\left(x\right)=\frac{K}{1+x},
\]
and substituting this in $t^{2}=\frac{d^{3}}{Gm_{2}}\:\Phi\!\left(\frac{m_{1}}{m_{2}}\right)$
or $t^{2}=\frac{d^{3}}{Gm_{1}}\:\Phi\!\left(\frac{m_{2}}{m_{1}}\right)$
we obtain 
\begin{equation}
t{}^{2}=\frac{Kd^{3}}{G\left(m_{1}+m_{2}\right)}\quad\mathrm{or}\quad t=C\sqrt{\frac{d^{3}}{G\left(m_{1}+m_{2}\right)}}.\label{eq:kepler}
\end{equation}
As before, (\ref{eq:kepler}) can be interpreted both as a quantity
relation and as a scalar relation. The quantity square root is well-defined
because $\left[\frac{d^{3}}{G\left(m_{1}+m_{2}\right)}\right]=\left[t\right]^{2}$.
The constants cannot be determined by dimensional analysis, but it
can be shown based on empirical data or theoretical considerations
that $K=4\pi^{2}$, $C=2\pi$.

It is worth pointing out that Bridgman considers only one solution
to this problem, namely 
\[
t=\frac{d^{\frac{3}{2}}}{G^{\frac{1}{2}}m_{2}^{\frac{1}{2}}}\:\phi\!\left(\frac{m_{2}}{m_{1}}\right);
\]
the corresponding relation here is $t^{2}=\frac{d^{3}}{Gm_{2}}\:\Phi_{1\!}\!\left(\frac{m_{1}}{m_{2}}\right)$.
Thus, Bridgman is not able to derive the functional equation $x\,\Phi\!\left(x\right)=\Phi\!\left(1/x\right)$
by setting $\Phi_{1}=\Phi_{2}$, and he does not obtain the much more
informative formula (\ref{eq:kepler}), derived here using nothing
but dimensional analysis and a natural symmetry assumption. 
\end{example}
\begin{example}
(Adapted from Barenblatt \cite{key-2}, pp. 13--14.) A \emph{Koch
snowflake} or \emph{Koch star} (named after the Swedish mathematician
Helge von Koch) is constructed as follows: Start with an equilateral
triangle. Then divide each side into three equal line segments and
replace the middle segment with the two other sides of an outwards-pointing
equilateral triangle with the middle segment as its base. All line
segments thus obtained can be subjected to the same operation, and
this is repeated as many times as desired, or \emph{ad infinitum}.

\includegraphics[bb=0bp 30bp 362bp 362bp,scale=0.5]{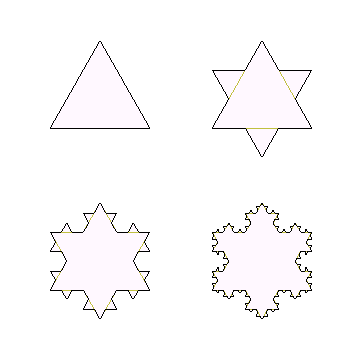}

\emph{Figure 1. Koch snowflake after $n=0,1,2,3$ iterations.}%
\footnote{ ``KochFlake''. Licensed under Creative Attribution--Share Alike
3.0 via Wikimedia Commons.%
}\emph{\bigskip{}
 }

It is clear that the length $\ell$ of the curve obtained after $n$
iterations can be expressed as a function $F$ of $n$ and the length
$d$ of one side of the original triangle. Equivalently, the length
of the curve can be expressed as a function of $d$ and $\eta,$ the
length of the line segments making up the curve after $n$ iterations;
thus $\ell=\Psi\!\left(d,\eta\right)$, where $\eta=f\!\left(d,n\right)$
and $\Psi$ is given by $\Psi\!\left(d,f\!\left(d,n\right)\right)=F\!\left(d,n\right)$.
We shall consider the relation $\ell=\Psi\!\left(d,\eta\right)$ in
order to illustrate a general point about dimensional analysis. The
simple dimensional matrix corresponding to this relation is 
\[
\begin{array}{c}
\\
L
\end{array}\left\Vert \begin{array}{ccc}
\left[\ell\right] & \left[d\right] & \left[\eta\right]\\
1 & 1 & 1
\end{array}\right\Vert ,
\]
and there are two possible dimensional models, 
\[
\begin{array}{c}
\left\Vert \begin{array}{c}
\left[\ell\right]\\
1
\end{array}\right|\begin{array}{c}
\left[d\right]\\
1
\end{array}\left|\begin{array}{c}
\left[\eta\right]\\
1
\end{array}\right\Vert \end{array},\qquad\begin{array}{c}
\begin{array}{c}
\left\Vert \begin{array}{c}
\left[\ell\right]\\
1
\end{array}\right|\begin{array}{c}
\left[\eta\right]\\
1
\end{array}\left|\begin{array}{c}
\left[d\right]\\
1
\end{array}\right\Vert ,\end{array}\end{array}
\]
with corresponding quantity relations $\ell=\eta\,\Phi_{1}\!\left(d/\eta\right)$
and $\ell=d\,\Phi_{2}\!\left(\eta/d\right)$.

Unfortunately, if we do not know $\Phi_{1}$ or $\Phi_{2}$ these
relations tell us nothing about how $\ell$ depends on $d$ (for a
fixed $\eta$) or on $\eta$ (for a fixed $d$). In this case, there
is no reason to assume that $\Phi_{1}=\Phi_{2}$, so we cannot find
a functional equation in the same way as in previous examples. There
is an alternative approach, however: to use only one of the two relations,
and derive an explicit expression for the corresponding function from
what we know about the procedure used to construct the Koch star.
It is intuitively clear that the length of the curve for fixed $\eta/d$
is proportional to $d$, so it is natural to try to determine the
function $\Phi_{2}$ in the relation $\ell=d\,\Phi_{2}\!\left(\eta/d\right)$.

To simplify, let us temporarily interpret $\ell$, $d$ and $\eta$
as scalars. In each step of the construction of the Koch star, a line
segment of length $\delta$ is replaced by four line segments of length
$\delta/3$. This means that after $n$ steps we have $\eta=d/3^{n}$,
so that 
\[
n=\log\left(d/\eta\right)/\log3,
\]
and that we also have $\ell=3d\left(4/3\right)^{n}$, so that 
\[
\ell=3d\left(e^{n\left(\log4-\log3\right)}\right)=3d\left(e^{\alpha\log\left(d/\eta\right)}\right)=3d\left(d/\eta\right)^{\alpha},
\]
where $\alpha=\frac{\log4-\log3}{\log3}\approx0.26$. Comparing this
expression with $\ell=d\,\Phi_{2}\!\left(\eta/d\right)$, where $\ell$,
$d$ and $\eta$ are quantities, we conclude that $\Phi_{2}\!\left(x\right)=3\left(\mu\!\left(x\right)\right)^{-\alpha}1_{Q}$,
giving the quantity relation 
\begin{equation}
\ell=d\left(3\left(\mu\!\left(\eta/d\right)\right)^{-\alpha}1_{Q}\right)=3\left(\mu\!\left(\eta/d\right)\right)^{-\alpha}d.\label{eq:ex7q}
\end{equation}
Incidentally, from (\ref{eq:ex7q}) and the relation $\ell=\eta\,\Phi_{1}\!\left(d/\eta\right)$
we can calculate $\Phi_{1}$, and we obtain $\Phi_{1}\!\left(x\right)=3\left(\mu\!\left(x\right)\right)^{1+\alpha}1_{Q}$,
so $\Phi_{1}\neq\Phi_{2}$ as anticipated.

Since $\mu\!\left(\eta/d\right)=\mu_{B}\!\left(\eta\right)/\mu_{B}\!\left(d\right)$
for any basis $B$ and $\mu_{B}\!\left(\ell\right)=\mu_{B}\!\left(d\right)\mu_{B}\!\left(\Phi_{2}\!\left(\eta/d\right)\right)$
for any $B$, the scalar relation corresponding to (\ref{eq:ex7q})
is 
\[
\mu_{B}\!\left(\ell\right)=\mu_{B}\!\left(d\right)\left(3\left(\mu_{B}\!\left(\eta\right)/\mu_{B}\!\left(d\right)\right)^{-\alpha}\right)=3\,\mu_{B}\!\left(d\right)^{1+\alpha}\mu_{B}\!\left(\eta\right)^{-\alpha}
\]
for any $B$, which can be written as 
\begin{equation}
\ell=3d^{1+\alpha}\eta^{-\alpha}.\label{eq:ex7s}
\end{equation}

The right-hand side of (\ref{eq:ex7s}) is a monomial with \emph{irrational
exponents}, multiplied by a constant. Note that such a monomial cannot
be obtained by dimensional analysis alone, since dimensional analysis
as understood here produces exponents which are integers or, equivalently,
rational numbers. As dimensional analysis ideally produces a relation
of the form 
\[
q^{k}=Kq_{1}^{k_{1}}\ldots q_{n}^{k_{n}}\quad\mathrm{or}\quad q=Cq_{1}^{k_{1}/k}\ldots q_{n}^{k_{n}/k},
\]
the scalar relation (\ref{eq:ex7s}) might give the impression that
only dimensional analysis is needed to derive it, but the quantity
relation (\ref{eq:ex7q}) reveals that this is not the case, although
(\ref{eq:ex7q}) also shows that the conclusion that $\ell=d\,\Phi_{2}\!\left(\eta/d\right)$,
based on dimensional analysis, can be separated from the conclusion
about $\Phi_{2}$, based on additional information specific to the
Koch star.\end{example}

\end{document}